 \documentclass[11pt,letterpaper]{amsart}
\usepackage[utf8]{inputenc}
\usepackage{amsfonts,amsthm,amsmath,amssymb,amscd,mathrsfs}
\usepackage{graphics}
\usepackage{indentfirst}
\usepackage{cite}
\usepackage{latexsym}
\usepackage[dvips]{epsfig}
\usepackage{color}
\usepackage{bookmark}

\usepackage{lmodern}

\usepackage{caption}

\newtheorem{theorem}{Theorem}[section]
\newtheorem{corollary}[theorem]{Corollary}
\newtheorem{lemma}[theorem]{Lemma}
\newtheorem{proposition}[theorem]{Proposition}
\newtheorem{problem}{Problem}[section]

\theoremstyle{definition}
\newtheorem{definition}[problem]{Definition}

\theoremstyle{remark}
\newtheorem{remark}[problem]{Remark}
\numberwithin{equation}{section}

\begin{document}

\title{A Classification Theorem for Steady Euler Flows}
\author{Tarek M. Elgindi}
\address{Department of mathematics, Duke University, Durham, NC, USA}
\email{tarek.elgindi@duke.edu}

\author{Yupei Huang}
\address{Department of mathematics, Duke University, Durham, NC, USA}
\email{yh298@duke.edu}
\author{Ayman Said}
\address{Laboratoire de Mathématiques de Reims (LMR-CNRS UMR9008), Reims, France}
\email{ayman.said@cnrs.fr}
\author{Chunjing Xie}
\address{School of mathematical Sciences, Institute of Natural Sciences, Ministry
of Education Key Laboratory of Scientific and Engineering Computing, and CMA-Shanghai,
Shanghai Jiao Tong University, 800 Dongchuan Road, Shanghai, China}
\email{cjxie@sjtu.edu.cn}

\date{}

\begin{abstract}
 Fix an analytic simply connected domain $\Omega\subset\mathbb{R}^2.$ We show that all analytic steady states of the Euler equations with stream function $\psi$ are either radial or solve a semi-linear elliptic equation of the form $\Delta \psi = F(\psi)$ with Dirichlet boundary conditions. In particular, if $\Omega$ is not a ball, then there exists a one to one correspondence between analytic steady states of the Euler equations and analytic solutions of equations of the form $\Delta \psi = F(\psi)$ with Dirichlet boundary conditions.

\end{abstract}


\maketitle

\setcounter{tocdepth}{1}
\tableofcontents
\section{Introduction and Main Results}\label{secintroduction}
The study of steady Euler flows is old and vast, the first such solutions appearing in Euler's original paper \cite{euler1757principes}. The interest in steady flows is natural, both from a mathematical and a practical point of view. From the mathematical side, steady flows play a distinguished role in the long-time dynamics of all solutions, as they do in any dynamical system. Indeed, the first step to understanding the long-time dynamics of any system is to have a deep and complete understanding of the steady solutions. While it is well known that the Euler equations admit an infinite-dimensional family of steady states on any two-dimensional domain, little is known about the global properties of the set of solutions.  The purpose of this work is to take a step in that direction.     

\subsection{The Euler Equations}
We begin by recalling the 2D Euler equations for incompressible flow on a simply connected bounded domain $\Omega\subset\mathbb{R}^2$:
\begin{equation}\label{EE}
\partial_t\omega + u\cdot\nabla\omega=0,
\end{equation}
\begin{equation}\label{BS}
u=\nabla^\perp \Delta^{-1}\omega.
\end{equation}
Here, $\omega$ is the vorticity of the fluid and $u$ is its velocity field. The symbol $\nabla^\perp$ refers to the rotated gradient $\nabla^\perp=(-\partial_2,\partial_1)$, while $\Delta^{-1}$ is the inverse of the Dirichlet Laplacian on $\Omega.$  It is a well-known fact that sufficiently smooth solutions to \eqref{EE}-\eqref{BS} enjoy global existence and uniqueness, though $\omega$ may lose its smoothness in the long-time limit, $t\rightarrow\infty.$ 

One of the most captivating aspects of the 2D Euler equation is in the long-time behavior of solutions.  An amazing fact, borne out of numerical and physical experiments, is that 2D Euler solutions starting from ``any old data'' tend to relax to much simpler states in the long-time limit. This is conjectured to occur, despite the conservative nature of the equation and the absence of a clear relaxation mechanism. Recent breakthroughs have been achieved in our understanding of this in certain \emph{perturbative} regimes, starting with the work \cite{bedrossian2015inviscid} and then the later works \cite{ionescu2020inviscid,NonlinearIonescu,masmoudi2024nonlinear}.  See also \cite{drivas2023twisting} for a recent work leveraging the classical stability theory of steady states to make some long-time statements. Understanding the mechanism for this type of behavior outside the perturbative regime appears to be well outside the reach of any of the available techniques in the field. In order to be able to study such large data phenomena, perhaps in the long-time limit, it is necessary to build a flexible theoretical framework that will allow us to study questions of a more global nature. For us, this begins with having a robust understanding of Euler steady states. 

\subsection{Steady States}

When studying steady solutions, it is helpful to introduce the stream function \[\psi=\Delta^{-1}\omega,\] in which case the steady Euler equations can be written as:
\begin{equation}\nabla^\perp\psi\cdot\nabla\Delta\psi:=\label{SEEBracket}\{\psi,\Delta\psi\}=0,\end{equation} where we note that $\psi$ is chosen to be zero on $\partial\Omega.$  Classically, there are three classes of solutions that have been considered: 
\begin{itemize}
\item Solutions for which $\psi$ is invariant under translation in a fixed direction $e$, called shear flows. 
\item Solutions for which $\psi$ is invariant under rotations about a fixed point $x_0,$ called radial flows. 
\item Solutions for which $\psi$ satisfies a semilinear elliptic equation:
\begin{equation}\label{SemilinearElliptic}\Delta \psi= F(\psi),\end{equation} for some function $F:\mathbb{R}\rightarrow\mathbb{R}.$
\end{itemize}
The presence of the first two is due to the symmetries of the Laplacian and the Poisson bracket. The third type is based on the simple observation that 
\[\{\psi, F(\psi)\}=F'(\psi)\{\psi,\psi\}=0,\] for any $F\in C^1.$ Interestingly, it was only recently shown that there exist steady states strictly\footnote{It is easy to find $C^\infty$ solutions consisting of two compactly supported radial bumps with respect to different centers that technically aren't of any of these types, though they are locally radial.} outside of these three types in the work \cite{gomez2021existence} with some regularity and with any \emph{finite} regularity in the work \cite{enciso2024smooth}. Both of these works rely on highly non-trivial bifurcation methods.

These three classes of steady solutions have a very nice structure. A quite robust linear stability theory has been developed for the first two, through the study of the Rayleigh equation \cite{drazin2004hydrodynamic,lin2003instability,tollmien1935allgemeines}. The stability theory for the third type can be investigated using properties of the linearized operator associated with \eqref{SemilinearElliptic}:
\[\mathcal{L}=\Delta-F'(\psi),\] see \cite{arnold1966geometrie,arnold1966principe,arnold2009topological, lin2004nonlinear,lin2004some}. Given the tools the community has developed over the years to study these three types of steady states, it would be great if we were able to show that these are the \emph{only} steady states out there. This, and more, has been established locally in some non-degenerate settings \cite{choffrut2012local,constantin2021flexibility, coti2023stationary, elgindi2022regular}  (here, non-degenerate means that the Schr\"odinger operator $\mathcal{L}$ above is invertible). 

In recent years, there has appeared a quite large literature on rigidity results for steady states. These results can be classified into two types: those making structural assumptions and those making smallness assumptions. In a series of papers, Hamel and Nadirashvili \cite{hamel2017shear,hamel2019liouville,hamel2021circular} established a number of rigidity results on symmetric domains assuming, for example, that the velocity field does not vanish on the domain. For example, they proved in \cite{hamel2021circular} that sufficiently regular steady states on a regular annulus, say $B_2(0)\setminus B_1(0)$, with non-vanishing velocity must be radial. Of course, as is exemplified by the Dirichlet eigenfunctions, the non-vanishing of the velocity field is essential to such a result. See also the work of Li, Lv, Shahgholian, and Xie \cite{Li2022Poiseuille} on the rigidity of Poiseuille flows in a strip and the work of Gui, Xie, and Xu\cite{Gui2024Euler} that studies the rigidity of steady Euler flows in terms of flow angles. Another interesting recent result, among others,  given by G\'omez-Serrano, Park, Shi, and Yao \cite{gomez2021symmetry} establishes the radial symmetry of solutions with compactly supported and non-negative vorticity. On the side of the results of rigidity under smallness conditions, there are now some results showing that any solution sufficiently close to some special shear flow must be a shear flow \cite{coti2023stationary,lin2011inviscid}. As an example, Lin and Zeng \cite{lin2011inviscid} proved that traveling-wave Euler solutions that are sufficiently close to the Couette flow on $\mathbb{T}\times [0,1]$ must be steady shear flows.  Finally, we mention the work of Constantin, Drivas, and Ginsberg \cite{constantin2021flexibility} that establishes results on both flexibility and rigidity in general domains, the main question being whether stable solutions inherit the symmetries of the domain. See also the works \cite{CastroLear1,CastroLear2, drivas2023islands, nualart2023zonal} for further results on flexibility and rigidity in various contexts. We now state our main results.
\subsection{Main Results}

\begin{theorem}\label{MainTheorem}
Let $\Omega$ be a simply connected and bounded domain of $\mathbb{R}^2$. Assume that $\psi$ is analytic on $\bar\Omega$, is constant on $\partial\Omega$, and satisfies the steady Euler equation \eqref{SEEBracket}. Let $[a,b]$ be the range of $\psi$ on $\bar{\Omega}.$ Then, one of the following must hold.
\begin{itemize}
\item $\Omega$ is a ball and $\psi$ is radial. 
\item There exists $F\in C([a,b])$ that is analytic on the open interval 
$(a,b)$, for which \[\Delta\psi=F(\psi).\] 
\end{itemize}
\end{theorem}
\vspace{-3mm}
\begin{remark}
   As $\psi$ is analytic in $\overline{\Omega}$ and $\partial \Omega$ is a level set of $\psi$, due to the structure theorem on the zero sets of real analytic functions (See Proposition \ref{structure of nodal set}), the boundary $\partial \Omega$ has to be a piecewise analytic curve.
\end{remark}
\begin{remark}
 We show moreover that $F$ is actually  H\"older continuous globally, admitting a convergent Puiseux series at the endpoints. In fact, the precise form of the Puiseux expansion of $F$ at $a$ and $b$, see Lemma \ref{lemma:Puiseux expression near 1}, allows us to define $\mathcal{L}=\Delta-F'(\psi)$ as a self-adjoint operator on suitable (weighted) spaces. In particular, it can be decomposed into an invertible Schr\"odinger operator plus a {relatively} compact perturbation. Hence the Fredholm {theory} can be applied to study its invertibility properties.  This is key to understanding the local structure of the set of steady states near arbitrary ones. We summarize this as an interesting consequence of the proof in the following corollary.
\end{remark}
\begin{corollary} For a given analytic, simply connected and bounded domain $\Omega$,
Suppose that $\psi\in C^\omega(\bar\Omega)$ (the set of real-analytic functions in $\bar{\Omega}$) takes a constant on $\partial \Omega$.  If \[\Delta\psi=F(\psi),\] and if $[a,b]$ is the range of $\psi$ in $\bar\Omega,$ then $F$ is analytic on $(a,b)$ and admits a (locally) convergent Puiseux series at the endpoints. 
\end{corollary}
The main novelty of Theorem \ref{MainTheorem}, compared with other results of rigidity in steady Euler flows mentioned above, is that we make no structural or smallness assumptions on $\Omega$ or on $\psi$. The results are also, in some sense, optimal, since \cite{enciso2024smooth} shows that we cannot relax the analyticity condition to finite smoothness, and Proposition \ref{main3} shows that the result cannot hold in general on multiply connected domains. 

\begin{remark}\label{analyticityremark} The analyticity of the velocity field is the major assumption of Theorem \ref{MainTheorem}. It would be nice if one could give a necessary and sufficient condition for the regularity of the velocity field so that Theorem \ref{MainTheorem} holds. On the other hand, for any $s>1$, there exists a single-variable compactly supported function in the Gevrey class $G^s$ (see \cite{krantz2002primer}).
Hence, we can use two compactly supported disjoint radial bumps to construct a steady flow whose stream function belongs to the Gevrey class $G^s$ with $s>1$, but is neither radial nor a solution to a global semilinear elliptic equation. Therefore, in terms of regularity, analyticity might be the optimal assumption for velocity fields.  
\end{remark}

\begin{remark}
    The simple connectedness of the domain is necessary. In Section \ref{sec: counterexample}, a counterexample is provided to illustrate that in a multiply connected domain, there exists a flow that can be neither described by a solution of global semilinear equation for the stream function nor a radial flow. 
\end{remark}

\begin{remark}\label{convex mirror}
    As an application of Theorem \ref{MainTheorem}, we can show that in a simply connected convex bounded domain, any analytic steady state whose vorticity is single-signed should inherit the mirror symmetry of the domain. In particular, any analytic steady flow with distinguished-signed vorticity in the disk is a radial flow. For more details on this property, one may refer to the Appendix \ref{application appendix}.
\end{remark}

\begin{remark}
To better understand the long-time behavior of unsteady flow, it would be desirable to establish a similar result in a broader class of steady states. An important fact that could help us to get closer is that \emph{every} steady state with $\psi\in C^{2,\alpha}$ has analytic Lagrangian trajectories \cite{chemin1992regularite,serfati1995structures, constantin2015analyticity}. In fact, this can be extended to all steady states with merely bounded vorticity. The main idea is that steady Euler solutions with bounded vorticity actually have some stratified regularity away from the critical points, as in the case of regular vortex patches \cite{serfati1994preuve, chemin1993persistance, bertozzi1993global}. We will present this result in a forthcoming paper.
\end{remark}

\subsection{Main Ideas}
We now move on to discuss some of the main ideas that go into the proof of Theorem \ref{MainTheorem}.
As we have mentioned above, the proof of Theorem \ref{MainTheorem} is based on a number of observations about the Poisson bracket, the zero sets of analytic functions, and over-determined elliptic problems. To clarify the main points of the argument, we will briefly sketch some of the main ideas here. 

\subsubsection{Some observations about $\{f,g\}=0$}\label{section1.4.1}
It is well known, using a straightening lemma, that if $f,g\in C^1$ and $\{f,g\}=0$ in a neighborhood of some $x_*\in\mathbb{R}^2$, while $\nabla f(x_*)\not=0,$ then we can write $g=F(f)$ for some $F\in C^1$ in a sufficiently small neighborhood of $x_*.$ This is a simple application of the inverse function theorem. It is also clear that this need not be the case globally. There are two important examples one should keep in mind in this regard, both of which we give in $\mathbb{T}^2$ (although the examples can be easily modified to work on the unit disk, for example).


\noindent \emph{Example 1: Non-Analytic $g$.} Take $f(x,y)=\sin(x)\sin(y).$ In this case, each \emph{regular} level set of $f$ has two connected components consisting of simple closed curves. We may take $g$ to be a bump function that is constant along the level sets of $f$ restricted to the first quadrant and also compactly supported inside the first quadrant. In this case, $g$ cannot be expressed as a function of $f$ globally on $\mathbb{T}^2$ (since $f$ takes the same values in the first and third quadrants, but $g$ takes different values there). We give two illustrative figures as in Figure \ref{fig:1}.
\begin{center}
\begin{figure}
\includegraphics[width=6cm, height=6cm]{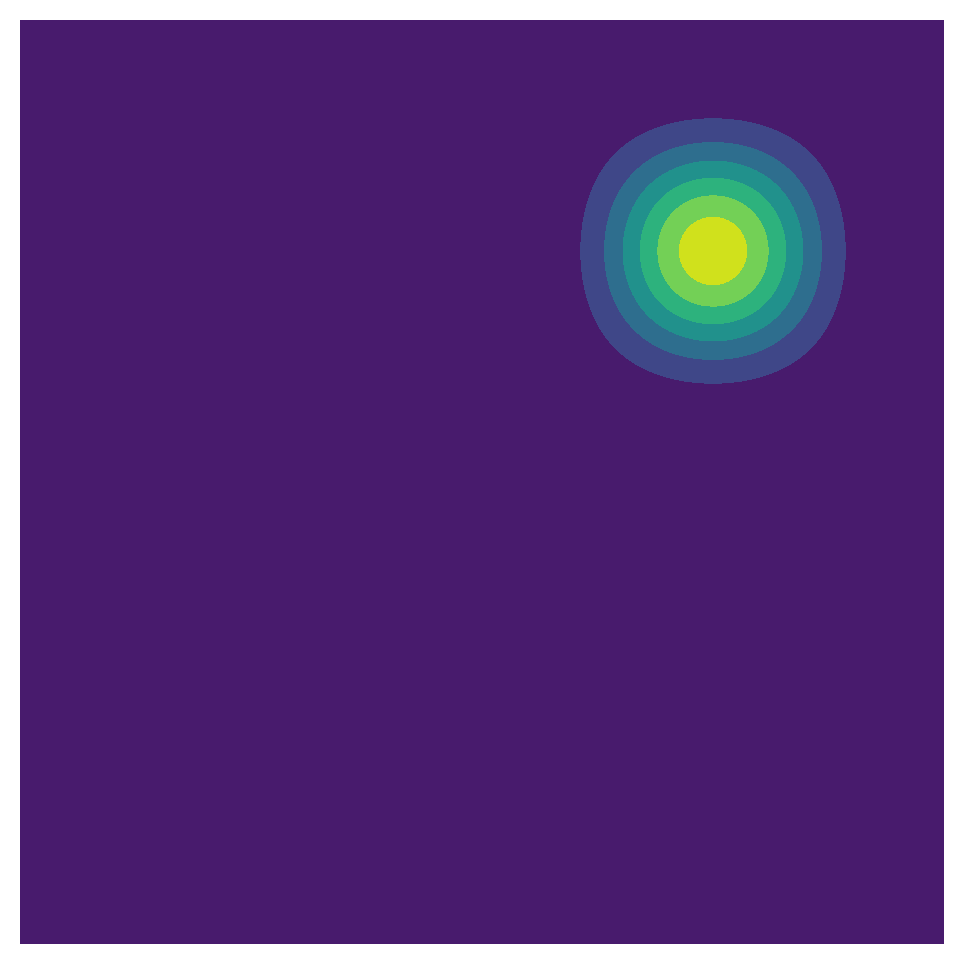}  
\includegraphics[width=6cm, height=6cm]{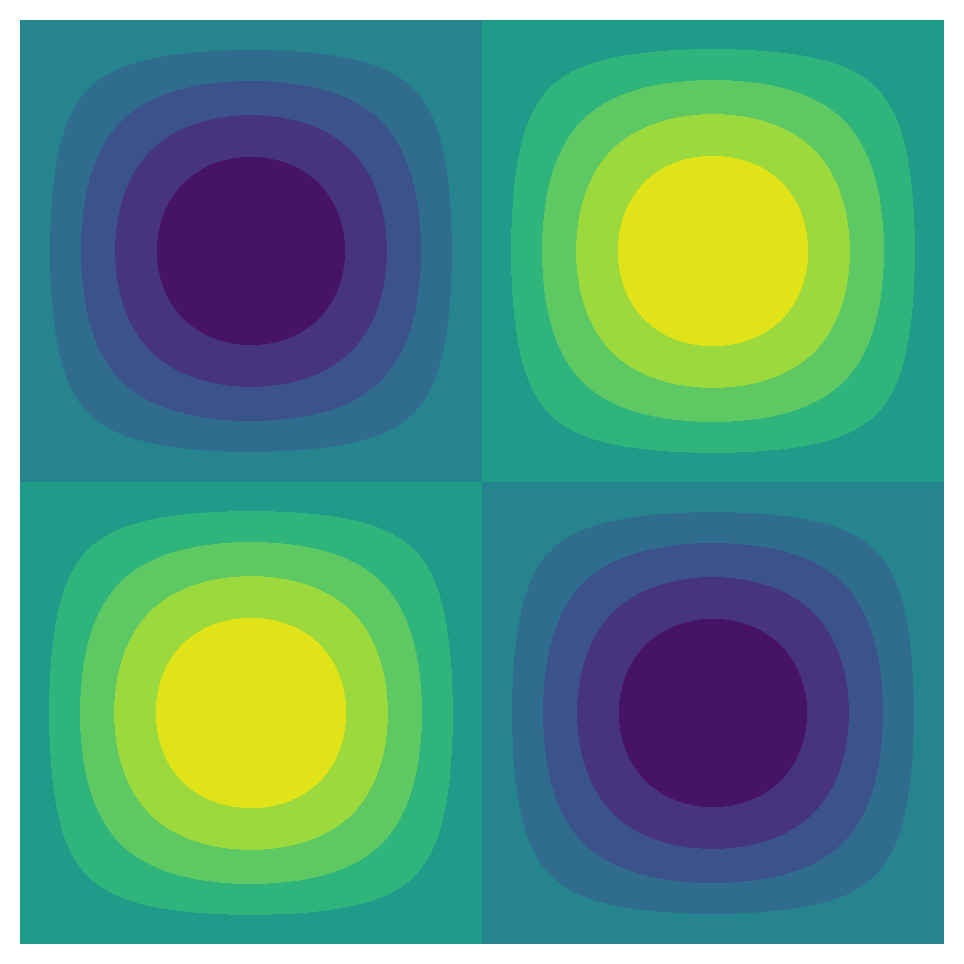}  
\captionof{figure}{2D Contour plot of $g$ and $f$ in Example 1}
    \label{fig:1}
\end{figure}
\end{center}


\noindent \emph{Example 2: A wall of critical points.} Take $f=\sin^2(x)\sin^2(y)$ and $g=\sin(x)\sin(y),$ then $g$ cannot be written as $F(f).$ In this case, $f$ and $g$ are both analytic; however, the set of critical points of $f$ partitions the domain into multiple cells surrounded by walls of critical points. We give two illustrative figures in Figure \ref{fig:2}, where the red lines are the set of critical points of $f$.

\begin{figure}
\centering
\includegraphics[height=6cm, width=6cm]{plotsin.png} 
\includegraphics[height=6cm, width=6cm]{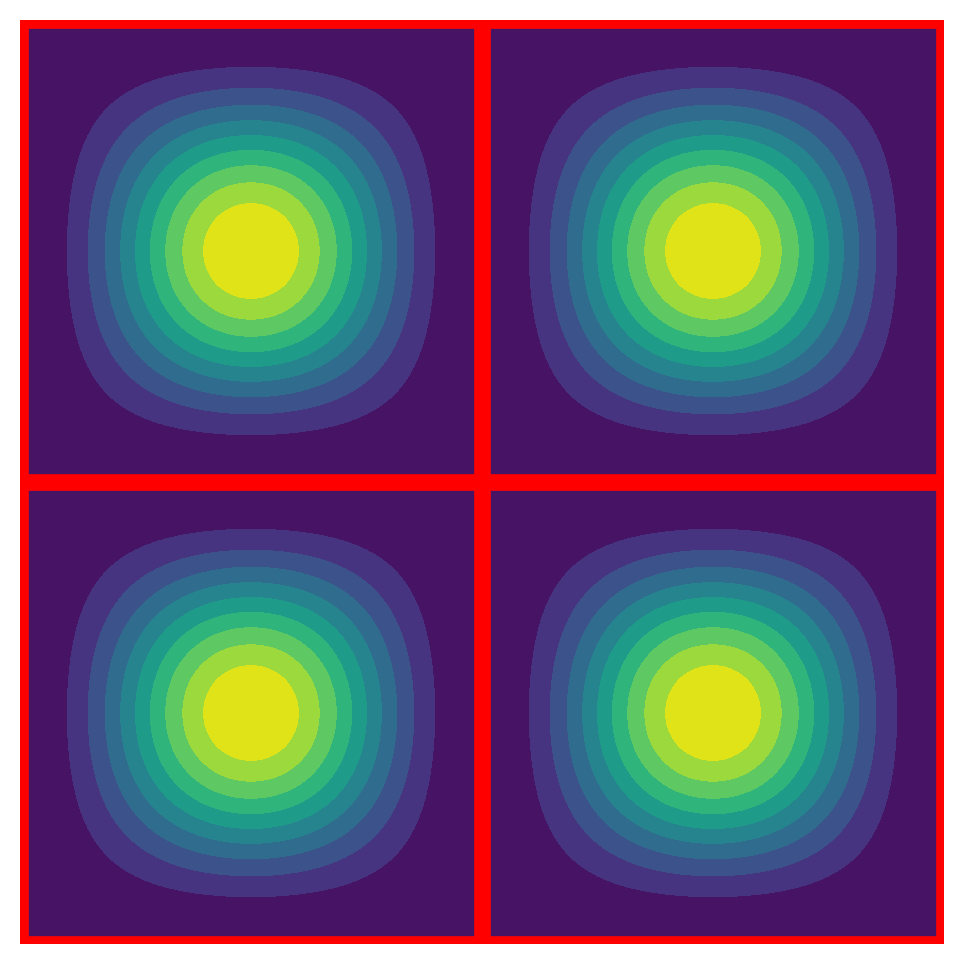}
\captionof{figure}{2D contour plot of $g$ and $f$ in Example 2}
\label{fig:2}
\end{figure}

It turns out that these examples encapsulate the only two obstructions to a global functional relation between $f$ and $g.$ We now state a simple consequence of our key proposition in this regard. 

\begin{lemma}\label{lemma:function relation bracket case}
Assume that $f,g$ are analytic functions on some open, connected domain $\Omega.$ Assume that $\{f,g\}=0$ and that the critical points of $f$ of even order are isolated. Then, there exists a continuous function $F:\mathbb{R}\rightarrow\mathbb{R}$ such that \[g=F(f).\]
\end{lemma}

\noindent In particular, in order that there does not exist a global $F$, $f$ must have ``walls'' of critical points of even order (see Definition \ref{Degree} for the precise meaning of critical point of even order).
\subsubsection{Properties of the set $\{\nabla\psi=0\}$ when $\{\psi,\Delta\psi\}=0$} 
Given the preceding discussion, the obstruction to having a global relation between $\Delta\psi$ and $\psi$ in a 2D Euler steady state is the existence of a curve of critical points of $\psi$. Since $\psi$ is analytic, we can consider the innermost such curve, $\Gamma$. {In other words, in the open region enclosed by $\Gamma$ (possibly with the aid of a portion of the boundary $\partial\Omega$),   called $\Omega_\Gamma,$ $\psi$ has only isolated critical points. In $\Omega_\Gamma$, $\psi$ will satisfy a semilinear elliptic equation with nonlinearity $F$ and the full gradient of $\psi$ vanishes on $\Gamma$.} Note that the curve $\Gamma$ may have singularities. {In fact, if $\psi$ cannot be written as a solution of \eqref{SemilinearElliptic} globally, one of the following three cases must happen:}

\begin{enumerate} \item [(i)] $\psi$ vanishes higher than quadratically on $\Gamma$;
 \item [(ii)] $\psi$ vanishes only quadratically (somewhere) on $\Gamma$ and the associated nonlinearity $F$ is not analytic at  $c=\psi(\Gamma)$;
 \item [(iii)] $\psi$ vanishes only quadratically (somewhere) on $\Gamma$ with $F$ analytic at $c=\psi(\Gamma).$
 \end{enumerate}
  In both cases (i) and (ii), because $F$ is necessarily non-analytic at $c=\psi(\Gamma)$, it turns out that $\psi-c$ is necessarily {single-signed} in $\Omega_{\Gamma}$--this is a key aspect of the argument. In case (i), we can apply a known result of Brock \cite{brock2000continuous} to deduce that $\psi$ must be radial; a nice feature of Brock's result \cite{brock2000continuous} is that it does not require the boundary of the domain to be smooth. The analysis in case (ii) is more difficult and will be discussed in more detail below.   We give an illustration depicting the potential complexity of the level-set in Figure \ref{Figure3}\footnote{The solid curves on the figure above are the set of critical points of $\psi$ in an elliptical domain. The boundary of the region $A$ demonstrates situation (i), where all points are critical points of degree higher than 2. The boundary of the region $B$ demonstrates the second case (ii), in the case where all points are quadratic. The regions $C$ and $D$ could represent the third case, where $F$ could be extended from $C$ to $D$, while $D$ falls into case (i) or (ii).}. 
\begin{center}
    \includegraphics[width=0.39\linewidth]{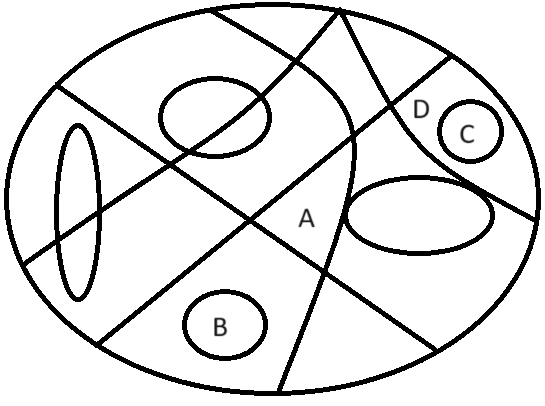}
    \captionof{figure}{Illustrations for a potential critical set of $\psi$}\label{Figure3}
\end{center} 

One snag in case (ii) is that we have to show that this curve $\Gamma$ is actually smooth and that $\psi$ is single-signed in its interior. 
For general analytic functions, it is well known that the zero set may be quite complicated and may contain multiple singular points of different types \cite{krantz2002primer,buffoni2003analytic}, even more so in high dimensions. Note, for example, that the set of critical points in Example 2 in Section \ref{section1.4.1} has a singularity at the origin. In Figure \ref{Figure3}, we have to show that case (ii) above actually corresponds to a region like $B$ in the figure (i.e., a region with a smooth boundary). A key observation here is that, when we have $\{\psi,\Delta\psi\}=0,$ the degree of vanishing of $\nabla\psi$ on each connected component of $\Gamma$ will necessarily be \emph{constant}. This is established by using the constancy of $\Delta\psi$ in the levels of $\psi$. Once the constancy of the degree is established, we use an elementary fact: singularities in the zero set of an analytic function occur \emph{only} at points where the degree of vanishing of the function changes. The constant degree of vanishing thus implies that $\Gamma$ is smooth.

{In the  case (iii), as $F$ is analytic in the boundary $\partial \Omega_{\Gamma}$, we can find an extension of $F$, such that $\psi$ would satisfy $\Delta \psi= F(\psi)$ in a domain containing $\Omega_{\Gamma}$. In practice, we do not work with case (iii) directly.  Instead, based on the discussions on the case (i) and case (ii), we show if $\psi$ is neither radial nor a global solution of \eqref{SemilinearElliptic}, then $\psi$ would have a nested family of curves of critical points. It then leads to a contradiction. }
\subsubsection{Moving Plane Arguments}
For case (ii) of the previous subsection, we see that if a steady state $\psi$ does not satisfy the global relation $\Delta\psi=F(\psi)$, up to subtracting a constant, we end up with a simply connected domain $\Omega_\Gamma$ with analytic boundary where \[\Delta\psi=F(\psi)\quad\text{in}\,\, \Omega_{\Gamma},\]  \[\psi=\nabla\psi=0 \quad \text{on}\,\, \partial\Omega_{\Gamma}.\]  We may also assume that $F(s)\approx \sqrt{s}$ as $s\rightarrow 0.$ As mentioned above, it is not difficult to show that $\psi$ cannot change sign in $\Omega_\Gamma$. By scaling, we can assume that $0\leq\psi\leq 1.$ Now, the overdetermined problem above seems similar to those considered by Serrin in his landmark paper \cite{serrin1971symmetry}. Unfortunately, a key difference here is that we do not have the Lipschitz smoothness of the nonlinearity $F,$ which is important for the moving plane method used in \cite{serrin1971symmetry}. Indeed, as we have already said, near $s=0,$ $F(s)\approx \sqrt{s}$; moreover,  $F$ need not be smooth as $s\rightarrow 1$ (the maximum value of $\psi$) either. To deal with the singularity as $s\rightarrow 0,$ we use the fact that we have a lower bound on the growth of $\psi$ in the direction normal to the boundary, since the vorticity is nonzero (and constant) on the boundary. Then we invoke a refinement of the Hopf lemma for elliptic problems on small domains under some growth conditions on the coefficients. This effectively deals with the singularity as $s\rightarrow 0$. Remarkably, it turns out that the potential singularity as $s\rightarrow 1$ always has the right sign for the maximum principle (see the proof of Lemma \ref{claim:coefficient estimate}). 

\subsection{Organization of the Paper}
In Section \ref{Bracket}, we establish some general results on analytic functions satisfying $\{f,g\}=0$ and recall a few results on zero sets of analytic functions. The main consequence of this is that, if an analytic Euler solution $\psi$ does not satisfy an equation of the type \eqref{SemilinearElliptic}, we are ensured the existence of a level set of $\psi$ on which the gradient vanishes. In Section \ref{LevelSet}, we prove a few results about these sets, using more fundamentally that $\psi$ is an Euler solution. In particular, we show the analyticity of the non-isolated portions of the critical points with quadratic vanishing. This sets us up with an overdetermined elliptic problem in a subdomain of the original domain. At the end of Section \ref{LevelSet}, we then establish the rigidity of solutions in a domain whose boundary is a curve of critical points with vanishing greater than quadratic. In Section \ref{Rigidity}, we {first} establish the rigidity of solutions with a suitably chosen curve of critical points with quadratic vanishing, which is based on certain regularity properties of the function $F$ (defined in a suitable subdomain), and the moving plane method. Then, we finish the proof of Theorem \ref{MainTheorem} in Section \ref{section:finish proof} via a contradiction argument. In Section \ref{sec: counterexample}, we give a counterexample to Theorem \ref{MainTheorem} in an analytic non-simply connected bounded domain, hence showing the optimality of the statement.

\section{Facts about $\{f,g\}=0$} \label{Bracket}

The purpose of this section is to establish Proposition \ref{KeyProposition} as well as to recall some basic facts about the zero sets of analytic functions. 

\subsection{Key Facts}

We start by giving a notation for the degree of vanishing of an analytic function at a point. 

\begin{definition}\label{Degree}
For any non-constant analytic function $f:\Omega\rightarrow\mathbb{R},$ and $x_*\in\Omega$, $D^k f(x_*)$ denotes the tensor for the $k$-th order derivatives of $f$ at $x_*$, and we define $d_{f}(x_*)$ to be the first $k\geq 1$ for which $D^k f(x_*)\not=0$. 
\end{definition}

We then introduce the notations for the set of critical points of $f$ and those with even degree of  vanishing:
\[\mathcal{C}(f)=\{x\in\Omega: d_f(x)>1\},\qquad \mathcal{C}_{e}(f)=\{x\in\Omega:  d_{f}(x)\in 2\mathbb{N}\},\]

\noindent and state the key proposition. 

\begin{proposition}\label{KeyProposition}
Assume that $\{f,g\}=0$ in some open connected set $\Omega\subset\mathbb{R}^2.$ Suppose that {$\Gamma_0$} is an analytic curve connecting $x,y\in\Omega$ and that {$\Gamma_0\cap\mathcal{C}_{e}(f)=\emptyset$. If $f(x)=f(y)$, then it must hold that $g(x)=g(y).$}
\end{proposition}

{Note that Lemma \ref{lemma:function relation bracket case} is a direct corollary of Proposition \ref{KeyProposition}.} To prove this proposition, let us first establish a generalization of the well-known straightening lemma. 

\begin{lemma}\label{lemma:Puiseux inverse}
Assume that $f,g$ are analytic and $\{f,g\}=0$ in a neighborhood of a point $x_*\in\Omega.$ Assume that $d:=d_f(x_*)\not\in 2\mathbb{N}.$ Then, there exists a continuous function $F$ so that $g=F(f)$ in a neighborhood of $x_*.$ In fact, $F(s)$ is analytic in the variable $(s-f(x_*))^{1/d}$.
\end{lemma}
\begin{remark}
The case $d=1$, which is that $\nabla f(x_*)\not=0,$ is the standard straightening lemma, which is valid under the only assumptions that $f$ and $g$ are $C^1$ in a neighborhood of $x_*.$
\end{remark}
\begin{remark}
Strictly speaking, we do not need the generalization to $d>1$ for the results proven later; however, it is conceptually important in order to understand the difference between the ``extendable'' and ``unextendable'' domains later. The result may also be of independent interest. It should be noted that, in this generalization, analyticity plays a key role, with the statement not holding in the smooth class.
\end{remark}

\begin{proof}[Proof of Lemma \ref{lemma:Puiseux inverse}]
Without loss of generality, by applying suitable translations and a rotation, set $x_*=0$, $f(x_*)=0,$ and $\partial_{1}^d f(0)>0.$ Assume that $d$ is odd. The point is that  $\partial_1 f(t,0)>0$ for $t\in [-\epsilon,0)\cup (0,\epsilon]$ for some $\epsilon>0$. In particular, the function $f(t,0)$ is real-analytic and invertible in $(-\epsilon,\epsilon)$. It follows that we can express its inverse using a Puiseux series for $t$ sufficiently small. In particular, there exists (a unique) $F$ that is analytic in the variable $s^{1/d}$ for which 
\[
g(t,0)=F(f(t,0)),\,\, \text{for every}\,\, t\in (-\frac{\epsilon}{2},\frac{\epsilon}{2}).
\]

Let $\kappa$ be the radius of convergence of the Puiseux series of $F$ at $0.$ We now take $\delta
<\frac{\epsilon}{2}$ be so small that $f(B_\delta(0))\subset (-\kappa,\kappa).$
We will show that 
\[
g(x)=F(f(x)),\,\,\text{for every}\,\, x\in B_\delta(0).
\]
Fix any $x\in B_{\delta}(0).$ First, take $x_1\in (-\delta,\delta)\setminus\{0\}$. By assumption, $(x_1,0)$ is not a critical point of $f$. It follows that there is a real analytic $G$ so that $g(z)=G(f(z)),$ for every $z$ in a neighborhood of $(x_1,0).$ In particular, there is an $\epsilon(x_1)>0$ such that for all $t\in (x_1-\epsilon(x_1),x_1+\epsilon(x_1)),$ we have
\[G(f(t,0))=F(f(t,0)).\] However, $f$ is monotone in that region. Therefore, $G=F$ in a neighborhood of $f(x_1,0).$ By analyticity, this means that $F\equiv G$ on their domains. 

At this point, we have shown that there is a ball around each point of the segment $\{(t,0):t\in(-\epsilon,\epsilon)\setminus\{0\}\}$ where $g=F(f).$ Now we will extend this to the whole ball $B_\delta(0).$ Fix any $x\in B_{\delta}(0)$ with $x_1\neq 0$. Without loss of generality, assume $x_2>0.$ Now consider the function $g(x_1,t),$ where $t\in [0,x_2].$ This is an analytic function of $t$, since $g$ is an analytic. Now, $F(f(x_1,t))$ is analytic and equal to $g$ in a neighborhood of $t=0$. Moreover, $F(f(x_1,t))$ is also analytic except at the possibly discrete set of points where $f(x_1,t)=0$, where $F(f(x_1,t))$ has a convergent Puiseux series (note that $f(x_1,0)\not=0$). It follows that $F(f(x_1,t))$ is actually analytic on $[0,x_2]$ and that $g(x_1,t)=F(f(x_1,t))$ for every $t\in [0,x_2].$ In particular, $g(x_1,x_2)=F(f(x_1,x_2)).$ 
\end{proof}

We are now in a position to prove Proposition \ref{KeyProposition}.

\begin{proof}[Proof of Proposition \ref{KeyProposition}]
 For any point $\tilde{x}\in \Gamma_0$, by applying Lemma \ref{lemma:Puiseux inverse} to $f$ and $g$, there are a neighborhood $U_{\tilde{x}}$ and a Puiseux function $F_{\tilde{x}}$ such that $g=F(f)$ in $U_{\tilde{x}}$. As the collection of the open set $U_{\tilde{x}}$ forms an open cover for the compact set $\Gamma_0$, we can choose $U_{x^1},\cdots, U_{x^m}$, such that $U_{x^i}\cap U_{x^{i+1}}\neq \emptyset$ and $\Gamma_0\subset \cup_{i=1}^{m}U_{x^i}$.
 In the case where $f$ is constant in $\Gamma_0$, as the pair $(g,f)$ satisfies a functional relationship in each $U_{x^i}$, $g$ is  constant in $U_{x^i}\cap \Gamma_0$. Notice that $\{U_{x^i}\}$ is a finite open cover for $\Gamma_0$, therefore, $g$ is constant in $\Gamma_0$. 
 
  In the case where $f$ is not constant in $\Gamma_0$, we claim that $F_{x^1}$ can be extended to be a Puiseux function $F$ in the range of $f(\Gamma_0)$ with $g=F(f)$ in $\Gamma_0$.   
 We first prove that $F_{x^1}$ can be extended to a Puiseux function $F$ in the range $f(U_{x^1}\cap \Gamma_0)\cup f(U_{x^2}\cap \Gamma_0)$, with $g=F(f)$ in $\Gamma_0 \cap (U_{x^{1}}\cup U_{x^{2}})$.
Here $F_{x^1}$ and $F_{x^2}$ are two Puiseux functions respectively in the range of $f(U_{x^1}\cap \Gamma_0)$ and $f(U_{x^2}\cap \Gamma_0)$, and the values of  $F_{x^1}$ and $F_{x^2}$ agree in $f(\Gamma_0 \cap  U_{x^1} \cap U_{x^2})$. As $f$ is not constant in $\Gamma_0$, the analytic assumption of $f$ and $\Gamma_0$ implies that $f(\Gamma_0\cap U_{x^1}\cap U_{x^2})$ contains an open interval in $\mathbb{R}$. Now due to the fact that $f(\Gamma_0\cap U_{x^1})$ and $f(\Gamma_0\cap U_{x^2})$ are connected, the condition that $F_{x^1}$ and $F_{x^2}$ agree in $f(U_{x^1}\cap U_{x^2})$ implies that the values of $F_{x^1}$ and $F_{x^2}$ agree in $f(\Gamma_0 \cap U_{x^1})\cap f(\Gamma_0 \cap U_{x^2})$. Hence, we can extend $F_{x^1}$ with the aid of $F_{x^2}$ to a Puiseux function $F$ in $f(U_{x^1})\cup f(U_{x^1})$ and $g=F(f)$ in $\Gamma_0 \cap (U_x^{1}\cup U_x^{2})$. Iterate the above procedure to extend $F_{x^1}$ to a Puiseux function $F$ in the range of $f(\Gamma_0)$ such that $g=F(f)$ in $\Gamma_0$. Hence the proof of the proposition is completed.
\end{proof}

\subsection{Zero Sets of Analytic Functions}
In this section, we will briefly discuss the structure of the zero set of a real analytic function in two variables. Let us first recall a classical result concerning the structure of the zero set of a real analytic function.
\begin{proposition}[See Corollary 2 in \cite{sullivan2006combinatorial}]\label{structure of nodal set} 
   Let $f$ be an analytic function near $x^0$  and $f(x^0)=0$, then one of the following three things must happen.
   \begin{enumerate}
        \item $x^0$ is an isolated point in the zero set for $f$.
       \item  $f\equiv 0$.
       \item The nodal set $\{f=0\}$ near $x^0$ is a collection of an even number of analytic curves that branch out at $x^0$. 
   \end{enumerate}
\end{proposition}
It is an immediate corollary of Proposition \ref{structure of nodal set} that for any given point $x$, which is not isolated in the nodal set for a non-constant analytic function $f$, there are at least two analytic curves $\Gamma_1$, $\Gamma_2$ branching at $x$ and consisting of nodal points of $f$. In most cases, we cannot glue $\Gamma_1$ and $\Gamma_2$ together to form a single analytic curve. For example, in the case $f(x_1,x_2)=x_1^2-x_2^3$, the nodal set has a cusp at the origin. However, in the case where $\Gamma_1$ consists of critical points of the same degree (in particular, if the degree does not jump up at the common point), $\Gamma_1\cup \Gamma_2$ is an analytic curve.
\begin{lemma}\label{lemma:no cusp}
    Let $x^0$ be a fixed point and we assume that there is an analytic curve $\Gamma_1:[0,1]\rightarrow \mathbb{R}^2$ with $\Gamma_1(0)=x^0$, and $d_{f}(x)=m\in [2, +\infty)$ for all $x\in \Gamma_1([0,1])$. Then the level set $\{x|f(x)=f(x^0)\}$ near $x^0$ is an analytic curve {$\Gamma_0$}. Moreover, we have $\Gamma_1\subseteq \Gamma_0$ and for any point $x\in \Gamma_0$, $d_f(x)=m$.
\end{lemma}
\begin{proof}
Without loss of generality, by applying a suitable rotation, we assume $x^0=0$, $f(0)=0$ and $\partial_1^{m}f(0)>0$. By applying the analytic inverse function theorem to $\partial_1^{m-1} f$ at $0$, we have the set $\{\partial_1^{m-1} f=0\}$ near $0$ is an analytic curve $\Gamma_0$ passing through $0$. The existence of the analytic curve $\Gamma_1$ consisting of critical points of degree $m$ that starts at $0$ ensures that for infinitely many points $\{x^i\}_{i=1}^\infty \subset \Gamma_1$ accumulate at $0
\in \Gamma_0$, we have \begin{equation}
    D^{k}f(x^i)=0,\,\, \text{for $1\leq k\leq m-1$ and $i\in \mathbb{N}$.  }
\end{equation} As $\Gamma_0$ and $f$ are analytic, it holds that \begin{equation}
    D^{k}f(x)=0, \,\,\text{for  $x \in \Gamma_0$, and $1\leq k\leq m-1$.} 
\end{equation}
Then due to the fundamental theorem of calculus, we have $f=0$ in $\Gamma_0$.

{Now that we have identified an analytic curve $\Gamma_0$ passing through the origin, which is a level set of $f$ and contains $\Gamma_1$, it remains to find an open neighborhood $\mathcal{O}$ of the origin such that 
\[
\{f=0\}\cap \mathcal{O} = \Gamma_0 \cap \mathcal{O}.
\]}  

We notice the condition  $\partial_1^{m}f(0)>0$ ensures that $(1,0)$ is not a tangent vector of $\Gamma_0$ at $0$. Now, since $f$ and $\Gamma_0$ are analytic, we can choose $\mathcal{O}:=(-\epsilon,\epsilon)^2$ to be a small cube near $0$ with the requirement that all points $x$ in $\mathcal{O}$ satisfy $\partial _1^{m} f>0$ and, in addition, for any point $x\in \mathcal{O}\setminus \Gamma_0$, there is a horizontal line $l_x$ connecting $x$ and a point $x_{\Gamma}$ in $\Gamma_0$. Now, as
\[
\partial _1 f(x_{\Gamma})=\cdots=\partial _1^{m-1}f(x_{\Gamma})=0\quad \text{and}\quad  \partial _1^{m} f>0\,\,\text{in}\,\, l_x,
\]
applying the mean value theorem to $f$ in $l_x$, we have $f(x)\neq 0$. 
\end{proof}
\section{Structure of the Set of Critical Points}\label{LevelSet}
From here and on, we fix a single analytic steady state $\psi,$ satisfying \eqref{SEEBracket} on $\bar\Omega.$ Towards proving Theorem \ref{MainTheorem}, we will also {\it {assume that  there is no $F$ for which $\psi$ solves \eqref{SemilinearElliptic} on all of $\Omega$}}. By Proposition \ref{KeyProposition}, as well as Proposition \ref{structure of nodal set}, it must thus be that $\psi$ has \emph{curves} of critical points.  The purpose of this section is to study these in more detail. To start, the equation 
\[\{\psi,\Delta\psi\}=0\] implies that $\Delta\psi$ is constant on any connected level set of $\psi.$ We thus have the following crucial lemma.

\begin{lemma}\label{lemma: degree 2}
Let {$\Gamma_0$} be a curve of points in $\mathcal{C}.$ Suppose that for some {$x_*\in\Gamma_0$}, $d_\psi(x_*)=2.$ Then, $d_\psi(x)=2$ for every $x$ in the connected component of $\{\nabla\psi=0\}$ containing {$\Gamma_0.$}
\end{lemma} 
\begin{proof}
{
We first choose $\delta$ small, such that the Hessian matrix $D^2\psi(z) \neq 0$ for all $z\in B_{\delta}(x_*)$. Based on Proposition \ref{structure of nodal set}, then there exists a $y\in B_{\delta}(x_*)\cap  \Gamma_0$ such that $\Gamma_0$ near $y$ is an analytic curve. Now pick $t$ and $n$ to be the tangent and normal direction of $\Gamma_0$ at $y$, respectively. The condition that $\Gamma_0$ consists of critical points of $\psi$ implies that $\partial_{tt}\psi(y)=\partial_{tn}\psi(y)=0$. Now, as $D^2\psi(y)\neq 0$, we have $\Delta \psi(y)=\partial_{tt}\psi(y)+\partial_{nn}\psi(y)\neq 0$. As $\{\psi,\Delta \psi\}=0$, we have that $\Delta \psi$ is constant along each connected component of $\{\nabla\psi=0\}$. In particular, we have $\Delta\psi(x) \neq 0$ and thus $d_\psi(x)=2$ for every $x$ in the connected component of $\{\nabla\psi=0\}$ containing $x_*$. This finishes the proof of the lemma.}
\end{proof}
Using Lemma \ref{lemma:no cusp},  we get the following corollary.
\begin{corollary}\label{DegreeTwo}
The points {in $\Omega$} for which $d_\psi(x)=2$ consist of a collection of isolated points or isolated analytic closed loops in $\Omega$.
\end{corollary}
\begin{proof}
{Recall that we are assuming that $\psi$ is analytic on $\bar\Omega.$ It follows then from Lemma \ref{lemma:no cusp} that if a curve $\Gamma_0$ of critical points of degree $2$ from the interior were to hit $\partial\Omega$, then the degree at the supposed intersection point, $x\in\partial\Omega$, would have to be at least $3$. This would contradict the constancy of $\Delta\psi$ along $\Gamma_0,$ as in Lemma \ref{lemma: degree 2}.} 
\end{proof}
The following result, which will be used later, is of a similar flavor.

\begin{lemma}\label{GeneralDegree}
{Let $d\geq 3$ and assume that $\Gamma_0$ be an analytic curve for which $d_{\psi}(x)=d,$ for all $x\in\Gamma_0$.} Then, for all $x\in\Gamma_0,$ 
\[d_{\Delta\psi}(x)=d-2.\]
\end{lemma}
\begin{proof}
{It is clear that $D^{i}\Delta \psi(x)=0$ for all $x\in \Gamma_0$ and $i\leq d-3$. Hence it suffices to check $D^{d-2}\Delta \psi(x) \not=0$ for all $x\in \Gamma_0$.}
Fix $x^0\in\Gamma_0$. By the rotational invariance of $\Delta,$ we may assume that the tangent vector to $\Gamma_0$ at $x^0$ is $(1,0).$ Now, for each $y\in\Gamma_0$ we have $D^{d-1}\psi(y)=0$. It follows that $D^{d-1}\partial_1\psi(x^0)=0$. {Now $D^{d}\psi(x^0)\not=0$ implies that $\partial_{2}^d\psi(x^0)\not=0$. Hence, one has $D^{d-2}\Delta \psi(x^0) \not=0$. This concludes the proof. }
\end{proof}

\begin{remark}
We note that it is actually possible to show that the degree of vanishing of an analytic Euler steady state has to be constant on any connected component of the set of critical points. This is shown in the thesis of the second named author \cite{yupeithesis}.
\end{remark}

\subsection{The Innermost Loop}
Since $\mathcal{C}(\psi)$ contains curves, $\Omega$ must be partitioned by the level set $\{\nabla\psi=0\}$ (this follows from the above discussions and Proposition \ref{structure of nodal set}). In particular, we can find at least one curve of critical points that is either a closed loop in $\Omega$ or starts and ends at $\partial\Omega$. Note that in the case where the curve starts and ends at $\partial\Omega$,  it must touch the boundary at two points (the two points could be the same).  Together with a boundary portion between these two points, the curve would still enclose a closed simply connected region.  Now, let $\Gamma$ be a curve with the property that the region it encloses (possibly with the aid of a portion of the boundary) contains no curves of critical points. Let us call this enclosed domain $\Omega_{\Gamma},$ which is necessarily simply connected.
 
 All of the analysis below will be restricted to $\Omega_{\Gamma}.$ Since the critical points of $\psi,$ contained in $\Omega_{\Gamma}$ must be isolated, we have that
\begin{equation}\label{SemilinearEllipticREAL}\Delta\psi=F(\psi)  \text{ in $\overline{\Omega_{\Gamma}},$ }\end{equation} for some continuous function $F$. By subtracting a constant, we may also assume that 
\[\psi=0\quad \text{on}\,\, {\partial\Omega_{\Gamma}}.\]
We will now show actually $\nabla\psi|_{\partial\Omega_{\Gamma}}=0,$ which is automatic when $\Gamma$ is a loop inside $\Omega$. {Note that if $\Gamma$ is not a loop in (the interior of) $\Omega$, then its endpoints must lie on $\partial\Omega$ (and may coincide).} 

\begin{lemma}\label{BoundaryCritical}
If $\Gamma$ touches $\partial\Omega,$ then the portion of $\partial\Omega$ between the two endpoints of $\Gamma$ in $\partial \Omega_{\Gamma}$ lies in $\mathcal{C}(\psi).$
\end{lemma}
\begin{proof}
First, we claim that if  $\Gamma$ touches $\partial\Omega$, then $d_\psi(x)\geq 3$ for all $x\in\Gamma$. Suppose that there is a point $x_0\in \Gamma\cap \Omega$ where $\psi$ vanishes at degree $2$. By Lemma \ref{lemma: degree 2}, all the points in $\Gamma$ have to be critical points of $\psi$ with vanishing degree $2$. However, just as in the proof of Corollary \ref{DegreeTwo}, we have shown that $\Gamma$ cannot touch the boundary. This leads to a contradiction.

Second, $\nabla \psi$ must vanish on $\partial\Omega\cap \partial \Omega_\Gamma$.
Indeed, take some $x\in\partial\Omega\cap \partial \Omega_\Gamma$, if $\nabla\psi(x)\not=0,$ then it follows from the inverse function theorem that $F$ from \eqref{SemilinearEllipticREAL} must be \emph{analytic} in a neighborhood of $0$ (as $\psi$ takes zero value in $\partial \Omega$). In addition, the condition that $d_\psi(x)\geq 3$ for all $x\in\Gamma$ shows {$\psi$ takes the value $0$ in $\Gamma$} and that $F(0)=0$. Now as $\Delta\psi=F(\psi)$ with $F$ analytic at $0$ and $F(0)=0$, considering the leading term of $F(\psi)$ near $\Gamma$, the vanishing degree of $\Delta \psi$ at $\Gamma$ is greater than or equal to the vanishing degree of $\psi$ at $\Gamma$. This would contradict Lemma \ref{GeneralDegree}. Hence, the proof of the lemma is complete.
\end{proof}

{Based on Lemmas \ref{BoundaryCritical} and \ref{DegreeTwo}, we have the following corollary concerning $\partial \Omega_{\Gamma}$.}
\begin{corollary}\label{TwoCases}
One of the following two cases must hold. \begin{itemize}
    \item Case 1: $\partial\Omega_{\Gamma}$ is {an analytic loop in $\Omega$} and contains only critical points of degree 2.
    \item Case 2: $\partial\Omega_{\Gamma}$ is a closed loop {in $\bar\Omega$} containing critical points of degree larger than 2. 
\end{itemize}
\end{corollary}

 Our goal will now be to show that, in both cases, $\psi$ is necessarily a radial function. We will first deal with Case 2, which will be solved essentially by applying existing results from the literature. Case 1 will be divided into two subcases based on whether the function $F$ appeared in \eqref{SemilinearEllipticREAL} is analytic at the value of the stream function on $\partial \Omega_\Gamma$. These two subcases will be dealt with in the following section using the moving plane method and contradiction argument, respectively. 

\subsection{Rigidity in Case 2}\label{sec3.2} 
In the previous subsection, we reduced the problem into studying a semilinear elliptic equation of Serrin type \cite{serrin1971symmetry}. More precisely, we found a simply connected domain $\Omega_{\Gamma}$ for which $\nabla\psi$ vanishes on $\partial\Omega_{\Gamma}$ and $\Omega_{\Gamma}$ contains only isolated critical points of $\psi$. As a consequence of Lemma \ref{lemma:function relation bracket case}, we deduced that $\psi$ solves \eqref{SemilinearEllipticREAL} in $\Omega_{\Gamma}$. In this section, we consider Case 2 of Corollary \ref{TwoCases}, where all points in $\partial\Omega_{\Gamma}$ are critical points with degree greater than or equal to $3$. Our goal is to establish the following proposition.  
\begin{proposition}\label{prop: rigidity in degenate case}
    Let $\Omega_{\Gamma}$ be a simply connected domain and $\psi$ be an analytic function in $\overline{\Omega_{\Gamma}}$ such that there is no non-isolated critical point of $\psi$ in $\Omega_{\Gamma}\cap \{\psi=0\}$. Now we further assume that there is a continuous function $F$, such that $\psi$ satisfies the following overdetermined problem for an elliptic equation.
    \begin{equation}
        \begin{aligned}
            &\Delta\psi(x)=F(\psi)(x),\text{ for $x\in \Omega_{\Gamma}$,}\\&
            \psi(x)=\nabla\psi(x)=\nabla^2\psi(x)=0, \text{ for all $x\in \partial \Omega_{\Gamma}$}.
        \end{aligned}
    \end{equation}
    Then $\psi$ is a radial function.
\end{proposition}
We first recall the following result based on the Steiner symmetrization method.
\begin{proposition}[From Theorem 5.2 and Theorem 6.1 in \cite{brock2000continuous} or Theorem 3.2 in \cite{ruiz2023symmetry}]\label{prop: steiner symmetry type}
    Let $\Omega$ be an open bounded domain and $\psi$ be a non-negative $C^2$ solution to \begin{equation}
    \left\{
    \begin{aligned}
          &\Delta{\psi}(x)=F(\psi)(x),\quad &x\in \Omega,\\
          &\psi(x)=0,\  \nabla \psi(x)=0, \nabla^2\psi(x)=0, \quad &x\in \partial \Omega.
    \end{aligned}
    \right.
    \end{equation} 
    Then $F$ being continuous implies that $\psi$ is locally radial in the sense that $$\mathbb{D}:=\{0<\psi<sup_{x\in \Omega}\psi(x)\}=A\cup \mathcal{C},$$
    where $A=\cup_{i}A_{i}$, $\mathcal{C}=\{\nabla \psi=0\}$ and $A_{i}$ are disjoint open annuli or balls and $\psi$ is radially symmetric in $A_{i}$.  
\end{proposition}
\begin{proof}[Proof of Proposition \ref{prop: rigidity in degenate case}]
Given Proposition \ref{prop: steiner symmetry type}, all we need to do is show that $\psi$ is {single-signed} in $\Omega_{\Gamma}.$ Towards a contradiction, assume that $\psi$ takes both positive and negative values in $\Omega_{\Gamma}.$ Simply by continuity, it follows that the set of zeros of $\psi$ in $\Omega_{\Gamma}$ is infinite. It follows from our choice of $\Gamma$ that at least one of those zeros is not a critical point of $\psi$ (since the critical set in $\Omega_{\Gamma}\cap\{\psi=0\}$ is discrete). It follows that $F$ is analytic in a neighborhood of zero. This contradicts Lemma \ref{GeneralDegree} (just as in the proof of Lemma \ref{BoundaryCritical}).
\end{proof}
\section{Moving Plane Arguments and the Proof of the Main Theorem}\label{Rigidity}
We now move on to prove the rigidity in Case 1 of Corollary \ref{TwoCases}, 
{where our basic assumption is that there is no $F$ for which $\psi$ solves \eqref{SemilinearElliptic}  in all of $\Omega$.} 
More precisely, we may assume that there exists a simply connected domain $\Omega_{\Gamma}$ with $\partial\Omega_{\Gamma}$ analytic and $\psi$ satisfies
\begin{equation}\label{SemilinearEllipticFinal}
        \begin{aligned}
            &\Delta\psi(x)=F(\psi)(x),\text{for $x\in \Omega_{\Gamma}$,}\\&
            \psi(x)=\nabla\psi(x)=0, \text{for all $x\in \partial \Omega_{\Gamma}$}.
        \end{aligned}
    \end{equation}
Moreover, by scaling, we may assume that $F(0)=1.$ In particular, \[\partial_{nn}\psi=1\quad \text{on}\,\, \partial\Omega_{\Gamma}.\] An important consequence is that $\psi$ increases in the normal direction of the boundary, in a tubular neighborhood of the boundary. In the case where $F$ is $C^1$ and $\psi$ is non-negative, we have the following rigidity result due to Serrin \cite{serrin1971symmetry} (see also the nice review \cite{nitsch2018classical}).
\begin{theorem}[Serrin, \cite{serrin1971symmetry}]
Assume that $\psi$ satisfies \eqref{SemilinearEllipticFinal} and that $\psi>0$ in $\Omega_{\Gamma},$ while $F$ is Lipschitz continuous. Then $\Omega_{\Gamma}$ is a ball and $\psi$ is radial. 
\end{theorem}

In our case, we cannot assume that $F$ is $C^1$ in general. In fact, there are two cases.

\begin{lemma}\label{lemma: Puisuex expression quadratic}
One of the following two cases must happen.
\begin{itemize}
\item $F(s)$ is not analytic at $s=0,$ in which case $\psi$ is non-negative in $\Omega_{\Gamma}.$ This is called the unextendable case. 
\item $F(s)$ is analytic at $s=0.$ This is called the extendable case.
\end{itemize}
\end{lemma}
\begin{proof}
{To prove the lemma, it suffices to check that $\psi$ is non-negative on $\Omega_{\Gamma}$ when $F$ is not analytic near $0$. In fact, the proof is quite similar to that in the proof of Proposition  \ref{prop: rigidity in degenate case}. Now we assume, towards a contradiction, that there is a point $x_*\in \Omega_{\Gamma}$ with $\psi(x_*)<0$. Since $\nabla \psi =0$ and $\partial_{nn} \psi =1$ on $\partial \Omega_{\Gamma}$, $\psi$ is positive in a tubular neighborhood of $\Gamma$ excluding $\Gamma$ itself. Hence, by continuity, it follows from $\psi(x_*)<0$ that there are infinitely many zeros of $\psi$ in $\Omega_{\Gamma}$. Since $F$ is not analytic at $s=0,$ the zeros of $\psi$ in $\Omega_\Gamma$ must be critical points. In particular, $\psi$ must have infinitely many critical points in $\Omega_\Gamma$. This contradicts the choice of $\Gamma$ as the innermost loop.   
}    
\end{proof}
\subsection{The Unextendable Case}\label{sec4.1}
This is the case where $F(s)$ is \emph{not} analytic in $s$ at $0$.
As mentioned in Lemma \ref{lemma: Puisuex expression quadratic}, $\psi$ must be non-negative in $\Omega_{\Gamma}$, since otherwise there would be infinitely many points where $\psi=0$ in $\Omega_{\Gamma}$. Without loss of generality, by scaling and multiplication by a constant factor, we assume that $\displaystyle \max_{\Omega_{\Gamma}}\psi =1$. Analogously as in the proof of Proposition \ref{prop: rigidity in degenate case}, the condition that $\psi$ does not have a non-isolated critical point in $\Omega_{\Gamma}$ implies that $F$ is analytic in $(0,1)$.
Our goal in this subsection is to prove the following proposition.
\begin{proposition}\label{prop:rigidity quadratic}
     Let $\Omega_{\Gamma}$ be a simply connected analytic domain, and $\psi$ be a non-negative analytic function in $\overline{\Omega_{\Gamma}}$. Assume $F\in C([0,1])$ is analytic on $(0,1).$   If $F$ is not analytic at $0$ and $\psi$ satisfies the over-determined problem \eqref{SemilinearEllipticFinal} for elliptic equation, then
     $\Omega_{\Gamma}$ is a ball and $\psi$ is radial.
\end{proposition}
\begin{remark}
Without loss of generality, we may assume $F(0)=1,$ since the case $F(0)=0$ has  already been covered by Proposition \ref{prop: rigidity in degenate case}.
\end{remark}
 For the proof of Proposition \ref{prop:rigidity quadratic}, it suffices to prove that $\psi$ is even symmetric to all directions $e\in \mathbb{S}^1$. Without loss of generality, by rotation,  we need only to prove the symmetry with respect to $e=(1,0)$. The key is to use the moving plane method to show that $\psi$ is even symmetric to $(1,0)$. We start with a series of lemmas concerning the regularity of $F$ and the critical set of $\psi$.
First, either $F$ is analytic at $1$ or $F$ is not smooth and admits a Puiseux expression near $1$.
\begin{lemma} \label{lemma:Puiseux expression near 1}
When the maxima of $\psi$ are critical points of degree $2k_0,$ for some $k_0\geq 1,$ $F(s)$ admits the following \emph{convergent} expansion, for  $0\leq 1-s\ll 1$, 
      \begin{equation}\label{eqn:Puiseux High order}
         F(s)=\sum_{k=k_0-1}^\infty a_k(1-s)^{\frac{k}{k_0}}, \text{ with $a_k$ real and $a_{k_0-1}<0$. } 
     \end{equation}   
\end{lemma}

\begin{remark}
The sign condition $a_{k_0-1}<0$ will play an important role in the moving plane argument. 
\end{remark}

\begin{proof}
    Fix a point $x^0\in \Omega_{\Gamma}$ such that $\psi(x^0)=1$, as $\psi$ attains the global maximum at $x^0$, we have $d_\psi(x^0)$ is an even number. In the case $d_\psi(x^0)=2$, without loss of generality, by a suitable rotation and translation, we assume $x^0=0$ and $\partial _{11} \psi(0)<0$. As $\partial _1\psi(0)=0$ and $\partial _{11} \psi(0)<0$, for $0\leq 1-s\ll 1$, the function $\psi(t,0)$ near $0$ has exactly two inverses $f_{+}(s)$ and $f_{-}(s)$, with $f_{+}(s)\geq 0$ and $f_{-}(s)\leq 0$.
    
    We show that $f_+$ and $f_-$ admit a Puisuex expansion near $1$. In particular, we will show that there is an analytic function $G$, such that $f_{+}(s)=G(\sqrt{1-s})$ and $f_{-}(s)=G(-\sqrt{1-s})$. 
    Indeed, given the Taylor expression of $\displaystyle \psi(t,0)=1+a_2t^2+\sum_{n\geq 3}a_{n}t^{n}$ with $a_2<0$, we have 
\begin{equation}
    \begin{aligned}
        \psi(f_{\pm}(s),0)=s\Leftrightarrow &f_{\pm}^2(s)\left(1+\sum_{k\geq 3} \frac{a_k}{a_2}f_{\pm}^{k-2}(s)\right)=\frac{s-1}{a_2} \\
        \Leftrightarrow
        &f_{\pm}(s) \sqrt{1+\sum_{k\geq 3} \frac{a_k}{a_2}f_{\pm}^{k-2}(s)}=\pm\sqrt{\frac{s-1}{a_2}}.
    \end{aligned}
\end{equation}
Note that the function $\displaystyle s\sqrt{1+\sum_{k\geq 3} \frac{a_3}{a_2}s^{k-2}}$ is an analytic function with a non-vanishing derivative at $0$, hence it admits an analytic inverse $H$. Therefore, we can correspondingly choose $G(s)=H(\frac{s}{\sqrt{-a_2}})$ so that $f_\pm(s)=G(\pm \sqrt{1-s})$. 

With the aid of $f_{+}$, $f_{-}$, and $G$, for $0\leq 1-s\ll 1$, we have $$F(s)=\Delta\psi(G(\sqrt{1-s}),0)=\Delta\psi(G(-\sqrt{1-s}),0).$$ The fact that $\Delta\psi(G(\sqrt{1-s}),0)$ agrees with $\Delta\psi(G(-\sqrt{1-s}),0)$, implies that {all the terms consisting of odd powers of $\sqrt{1-s}$ vanish in the Puiseux expansion of $F$ near $1$.} Therefore, $F$ is analytic near $1$.

Now in the case where $d_{\psi}(x^0)=2k_0\geq 4$, again assume that $x^0=0$ and $\partial _{1}^{2k_0}\psi<0$, for $0\leq 1-s\ll 1$, the function $\psi(t,0)$ near $0$ has  exactly two inverses $f_{+}(s)$, $f_{-}(s)$. Similarly, as in the case where $d_{\psi}(x^0)=2$, it follows that $F$ near $1$ has the Puiseux expansion.  More precisely, for $0\leq 1-s \ll 1$, we have \eqref{eqn:Puiseux High order}. This finishes the proof of the lemma.
\end{proof}
Analogously, by taking a normal line segment through a boundary point $x\in \partial \Omega_{\Gamma}$, we have $F(s)$ is analytic in the variable $\sqrt{s}$ near $0$. Together with the condition that $F$ is not analytic near $0$, we have the following Puiseux expansion of $F$ near $0$.
\begin{lemma}\label{lemma: F Puiseux expression near 0}
   $F$ near $0$ admits a convergent Puiseux series expression: for $0\leq s\ll 1$, one has
\begin{equation}
F(s)=\sum_{k=0}^\infty a_ks^{\frac{k}{2}}, 
\end{equation}
where  $a_k\in \mathbb{R}$, $a_0=1$ and $a_{k_0}\neq 0$ for an odd number $k_0$.
\end{lemma}
Moreover, arguing similarly as in the case $d_{\psi}(x^0)=2$ in the proof of Lemma \ref{lemma:Puiseux expression near 1}, the condition $F$ is not analytic at $0$ implies that $\psi$ is positive in $\Omega_{\Gamma}$.  Together with the fact \[\partial_{nn} \psi|_{\partial \Omega_{\Gamma}}=1,\] we can show that $\psi$ has a uniform lower bound in $\Omega_{\Gamma}$.
\begin{lemma}\label{Lemma:Lower bound of psi}
    There exists a constant $C>0$ such that for all $x\in \Omega_{\Gamma}$, \begin{equation}\label{eqn:Lower bound of psi}
        \frac{dist(x,\partial \Omega_{\Gamma})^2}{C}\leq \psi(x) \leq C dist(x,\partial \Omega_{\Gamma})^2.
    \end{equation}
\end{lemma}
Now we are ready to apply the moving plane method to show that $\psi$ is even symmetric in the direction $e=(1,0)$.
Without loss of generality, by translation, we fix $\displaystyle \inf_{x\in \Omega_{\Gamma}}x_1=0$. For $\lambda>0$, define
\[
H_{\lambda}:=\{x_1=\lambda\},\,\, 
\Omega_{\lambda}:=\{x_1<\lambda\}\cap \Omega_{\Gamma},\text{ and } \pi_{\lambda}(x_1,x_2):=(2\lambda-x_1,x_2). 
\]
We now reflect $\psi$ and $\Omega_\lambda$ through the hyperplane $H_{\lambda}$ as follows:  
$$\  \psi_{\lambda}(x):=\psi(\pi_{\lambda}(x)), \,\, \breve{\Omega}_{\lambda}:=\pi_{\lambda}(\Omega_{\lambda}).
$$ 
We are going to show $\psi_{\lambda}$ is always less than or equal to $\psi$ in $\breve{\Omega}_{\lambda}$ before $\partial \breve{\Omega}_{\lambda}$ touches the physical boundary $\partial \Omega_{\Gamma}$ tangentially. We first define the reflection monotonicity property for $\psi$ at level $\lambda$.
\begin{definition}
    We say $\psi$ has the reflection monotonicity property at the level $\lambda$, if $\breve{\Omega}_{\lambda}\subset \Omega_{\Gamma}$ and $\psi\geq \psi_{\lambda}$ in $\breve{\Omega}_{\lambda}$.
\end{definition}
Since $\partial \Omega_{\Gamma}$ is analytic, for $\lambda\ll 1$, $\breve{\Omega}_{\lambda}\subset \Omega_{\Gamma}$. Moreover, as $\nabla\psi=0$ and $\partial_{nn}\psi=1 $ in $ \partial \Omega_{\Gamma}$,  $\psi$ satisfies the reflection monotonicity property at the level $0< \lambda\ll 1$ . 
 We will now prove that the reflection monotonicity property persists unless $\partial \breve{\Omega}_{\lambda}$ hits $\partial \Omega_{\Gamma}$ tangentially.
\begin{lemma}\label{lemma:interior moving} 
     Define $$\lambda_0=\sup \{\lambda| \text{$\psi$ satisfies the reflection monotonicity property at level $\lambda$}\}.$$
     We have $h_{\lambda_0}=\psi-\psi_{\lambda_0} \geq 0$ in $\breve{\Omega}_{\lambda_0}$.
     Moreover, at least one of the following occurs when $\lambda=\lambda_0$. 
    \begin{itemize}
    \item{Internal tangency.} $\partial \breve{\Omega}_{\lambda_0}$ is tangent to $\partial \Omega_{\Gamma}$ at some point $x^0 \notin H_{\lambda_0}$.
    \item {Boundary tangency.}
    $\partial \breve{\Omega}_{\lambda_0}$ is tangent to $\partial \Omega_{\Gamma}$ at a point $x^0 \in H_{\lambda_0}$. In particular, the normal direction of $\partial \Omega_{\Gamma}$ at $x^0$ is $(0,1)$.
    \end{itemize}
\end{lemma}
\begin{center}
\begin{figure}
\includegraphics[width=6cm, height=6cm]{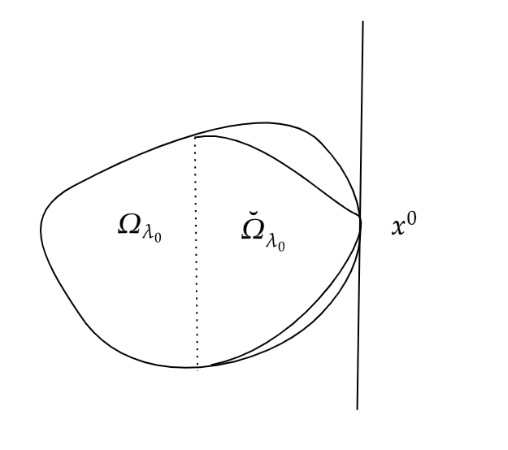} 
\includegraphics[width=6cm, height=6cm]{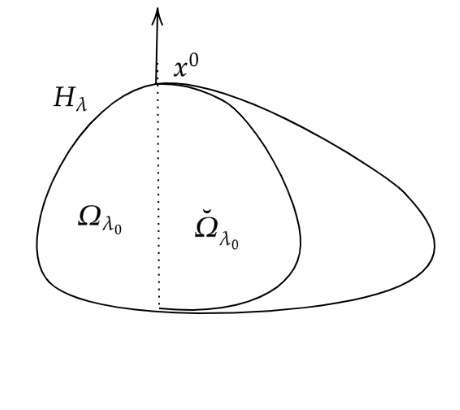} 
\captionof{figure}{Internal tangency and boundary tangency}  
\end{figure}
\end{center}
The proof of Lemma \ref{lemma:interior moving} is based on the observation that as long as $\breve{\Omega}_{\lambda}\subset \Omega_{\Gamma}$, the difference $\psi-\psi_{\lambda}$ satisfies an elliptic equation where the coefficient has a uniform estimate.
\begin{lemma}\label{claim:coefficient estimate}
    For all $\lambda>0$, such that $\breve{\Omega}_{\lambda}\subset \Omega_{\Gamma},$ $h_{\lambda}:=\psi-\psi_{\lambda}$ satisfies an elliptic equation \begin{equation}
        \Delta h_{\lambda}=c_{\lambda}(x)h_{\lambda},
    \end{equation}
    where the  coefficient $c(x)$ satisfies the following uniform estimate
    \begin{equation}\label{eqn:Upper bound singular coefficent 1}
        c_{\lambda}^{-}(x):=\max\{-c_\lambda(x), 0\}\leq \frac{M}{dist(x,\partial \breve{\Omega}_{\lambda})}.
    \end{equation} 
\end{lemma}
\begin{proof}
    Due to the reflection symmetry of the semilinear elliptic equation, $\Delta\psi_{\lambda}=F(\psi_{\lambda})$, and thus \begin{equation}
    \Delta h_{\lambda}=\frac{F(\psi)-F(\psi_{\lambda})}{\psi-\psi_{\lambda}}h_{\lambda}:=c_{\lambda}(x)h_{\lambda}.
\end{equation} 
\par In the case where $F$ is analytic at $1$, as $F$ is already analytic in $(0,1)$, from  Lemma \ref{lemma: F Puiseux expression near 0}, there exists a function $G \in C^1([0,1])$ and a constant $d$ such that $F(s)=d\sqrt{s}+G(s)$ ($d$ might be zero, which corresponds to the case where the Puiseux expansion of $F$ at $0$ has no $\sqrt{s}$ term). Now we have \begin{equation}
    c_{\lambda}=\frac{d}{\sqrt{\psi_{\lambda}}+\sqrt{\psi}}+\frac{G(\psi)-G(\psi_{\lambda})}{\psi-\psi_{\lambda}}
\end{equation}
It follows from the lower bound of $\psi$ in the estimate \eqref{eqn:Lower bound of psi} that one has \begin{equation}
    \psi_{\lambda}\geq \frac{dist(x,\partial \breve{\Omega}_{\lambda})^2}{C}.
\end{equation}
This, together with the fact that $\psi$ and $\psi_{\lambda}$ enjoy a uniform $C^1$ upper bound, yields $|c_{\lambda}|\leq \frac{M}{dist(x,\partial \breve{\Omega}_{\lambda})}$, where $M>0$ is \emph{independent} of $\lambda$.
\par In the case where $F$ is not analytic near $1$, since $F$ is analytic in $(0,1)$, Lemmas \ref{lemma:Puiseux expression near 1} and \ref{lemma: F Puiseux expression near 0} show that there exist an integer $k_0\geq 2$, a negative constant $d_1<0$, a constant $d_2$, and a function $G\in C^{1}([0,1])$, such that 
\[
F(s)=d_1(1-s)^{\frac{k_0-1}{k_0}}+d_2\sqrt{s}+G(s).
\]
We now have the corresponding expression for $c_{\lambda}$
\begin{equation}\label{eqn:expression of c}
    c_{\lambda}=d_1\frac{(1-\psi)^{\frac{k_0-1}{k_0}}-(1-\psi_{\lambda})^{\frac{k_0-1}{k_0}}}{\psi-\psi_{\lambda}}+\frac{d_2}{\sqrt{\psi_{\lambda}}+\sqrt{\psi}}+\frac{G(\psi)-G(\psi_{\lambda})}{\psi-\psi_{\lambda}}.
\end{equation}
 As $d_1<0$, the term $d_1\frac{(1-\psi)^{\frac{k_0-1}{k_0}}-(1-\psi_{\lambda})^{\frac{k_0-1}{k_0}}}{\psi-\psi_{\lambda}}$ is non-negative. Hence, similar to the case where $F$ is analytic at $1$,  we have $c_{\lambda}^{-}\leq \frac{M}{dist(x,\partial \breve{\Omega}_{\lambda})}$, for a positive $M$ independent of $\lambda$.
\end{proof}
We now prove Lemma \ref{lemma:interior moving} by the contradiction argument.
\begin{proof}[Proof of Lemma \ref{lemma:interior moving}]
  Due to the continuity of $\psi$, $h_{\lambda_0}:=\psi-\psi_{\lambda_0}\geq 0$ in $\breve{\Omega}_{\lambda_0}$. Now we assume that at the level of $\lambda_0$, none of the internal tangency or boundary tangency occurs, in particular, there is a $\delta^*>0$, such that for $\lambda_0\leq \lambda<\lambda_0+\delta^*$, $\breve{\Omega}_{\lambda}\subset \Omega_{\Gamma}$. We will show that there is a $\delta \in(0,\delta^*)$ such that $\psi$ satisfies the reflection monotonicity at level $\lambda\in[\lambda_0,\lambda_0+\delta]$.
  
  We start by proving that $h_{\lambda_0}$ is positive in $\breve{\Omega}_{\lambda_0}$. By Lemma \ref{claim:coefficient estimate}, $h_{\lambda_0}$ is a nonnegative solution to 
  \begin{equation}\label{problemlambda0}
  \Delta h_{\lambda_0}=c_{\lambda_0}(x)h_{\lambda_0}\,\, \text{with}\,\, c_{\lambda_0}^{-}\leq \frac{M}{dist(x,\partial \breve{\Omega}_{\lambda_0})}.
  \end{equation}
  Now suppose that there is a point $x^0\in \breve{\Omega}_{\lambda_0}$ such that $h_{\lambda_0}(x^0)=0$, we automatically have $\nabla h_{\lambda_0}(x^0)=0$. Then the Proposition \ref{Hopf Lemma 1} in the appendix shows that in a small ball whose boundary contains $x^0$, $h_{\lambda_0}$ is identically zero. Due to the analytic regularity of $\psi$, it further implies that $\psi_{\lambda_0}=\psi$. Hence $\Omega_{\Gamma}$ is even symmetric with respect to $H_{\lambda_0}$ and we get a contradiction.
 Therefore, it holds that 
 \begin{equation}\label{eqn: oth bound}
      h_{\lambda_0}(x)>0, \text{ for all $x\in \Breve{\Omega}_{\lambda_0}$.}
  \end{equation}
  
  Analogously, for any point $x\in H_{\lambda_0}\setminus \partial \Omega_{\Gamma}$, we may apply Proposition \ref{Hopf Lemma 1} to a small ball that hits the  boundary of $\partial\breve{\Omega}_{\lambda_0}$ tangentially at $x$, and conclude that \begin{equation}\label{eqn:linear bound}
      \partial_{1}h_{\lambda_0}(x)>0, \text{ for all $x\in H_{\lambda_0} \setminus \partial \Omega_{\Gamma}$},
  \end{equation} 
  unless $h_{\lambda_0}\equiv 0.$
  
  In the rest of the proof, we will show that there is a $\delta>0$ such that $\psi$ satisfies the reflection monotonicity property at level $\lambda\in [\lambda_0,\lambda_0+\delta].$
First, as $\breve{\Omega}_{\lambda_0}$ doesn’t hit the boundary $\partial \Omega_{\Gamma}$ tangentially, for all $\lambda\in [\lambda_0,\lambda_0+\delta^*]$, $ H_{\lambda}\cap \partial \Omega_{\Gamma}$ consists of two points $x_{\lambda}^{1}$,  $x_{\lambda}^{2}$($x_{\lambda}^{1}$ is above $x_{\lambda}^{2}$). Moreover, {as $\breve{\Omega}_{\lambda_0}$ doesn’t hit the boundary $\partial \Omega_{\Gamma}$ tangentially at the points $x_{\lambda}^{1}$ and $x_{\lambda}^{2}$, the condition that $\partial_{nn}\psi =1, \psi_{tt}\psi=\psi_{tn}=0$ ($t,n$ here represent the tangent and normal direction, respectively) for all points in $\partial \Omega_{\Gamma}$ implies that } there is an $\epsilon_0>0$ such that 
\[
\partial _{12}h_{\lambda_0}(x_{\lambda_0}^1)<-\epsilon_0\,\,\text{and}\,\, \partial _{12}h_{\lambda_0}(x_{\lambda_0})(x_{\lambda_0}^2)>\epsilon_0.
\]
As $\partial_{12}h_{\lambda}(x_{\lambda}^1)$ and $\partial_{12}h_{\lambda}(x_{\lambda}^2)$ are uniformly continuous in $\lambda$, there exist a  $\delta_1>0$  such that \begin{equation}\label{eqn: quadratic1 endpoint}
\begin{aligned}
     \partial _{12}h_{\lambda}(x_{\lambda}^1)<-\frac{1}{2}\epsilon_0 \,\,\text{and}\,\,\partial _{12}h_{\lambda}(x_{\lambda}^2)>\frac{1}{2}\epsilon_0, \text{ for $\lambda_0\leq \lambda\leq \lambda_0+\delta_1$}.
\end{aligned}
\end{equation}
As $x_{\lambda}^1$ is a critical point of $\psi$ and $\psi_\lambda$, we have 
\begin{equation}\label{eqn: linear endpoint}
    \nabla h_{\lambda}(x_{\lambda}^{1})=0.
\end{equation} 
Moreover, by symmetry, it holds that
\begin{equation}\label{eqn:quadratic endpoint 2}
    \partial_{11}h_{\lambda}(x_{\lambda}^{1})=\partial_{22}h_{\lambda}(x_{\lambda}^{1})=0.
\end{equation}
 Due to the smoothness of $\partial \Omega_{\Gamma}$ and that the boundary tangency doesn't occur when $\lambda=\lambda_0$, there exists an $\epsilon>0$, such that the direction $(1,-\epsilon)$ is always pointing to the exterior of $ \Breve{\Omega}_{\lambda}$ at $x_{\lambda}^1$ for all $\lambda\in [\lambda_0,\lambda_0+\delta]$(we may decrease the value of $\delta_1$ when it is necessary). Now
 since $\partial^3h_{\lambda}$ is uniformly bounded, from \eqref{eqn: quadratic1 endpoint}, \eqref{eqn: linear endpoint}, and \eqref{eqn:quadratic endpoint 2},  there exists an $r_1>0$ such that  $h_{\lambda}>0$ in $B_{r_1}(x_{\lambda}^1)$ for $\lambda\in [\lambda_0,\lambda_0+\delta_1]$. 
 Analogously,  $h_{\lambda}\geq 0$ in $B_{r_1}(x_{\lambda}^2)$ for $\lambda\in [\lambda_0,\lambda_0+\delta_1]$  (we decrease the value of $\delta_1$ when necessary).
 It follows from the estimate \eqref{eqn:linear bound} that there is an $\epsilon_1>0$ such that we have $\partial_1 h_{\lambda_0}(x)>\epsilon_1$ for all $x\in H_{\lambda_0} \setminus (B_{r_1}(x_{\lambda_0}^1)\cup B_{r_1}(x_{\lambda_{0}}^2))$. Similar to  the previous case,  we can find  $r_2>0$ and  $\delta_2 \in(0,\delta_1)$, such that for all $\lambda\in (\lambda_0, \lambda_0+\delta_2)$, $h_{\lambda}\geq 0$ in the set 
 \begin{equation*}
 T_{\lambda}:=\{x|x\in \Breve{\Omega}_{\lambda},\, dist(x,H_{\lambda})<r_2\} \setminus (B_{r_1}(x_{\lambda}^1)\cup B_{r_1}(x_{\lambda}^2)).
\end{equation*}

 In summary, we show there are a uniform $r>0$ and $\delta>0$, such that if $\lambda_0\leq  \lambda\leq \lambda_0+\delta$, we have $h_{\lambda}\geq 0$ in $N_{\lambda}:=\{x|x\in \breve{\Omega}_{\lambda}, dist (x,H_{\lambda})<r\}$, which is a tubular neighborhood of $H_{\lambda}$.
 Note by \eqref{eqn: oth bound} and Lemma \ref{Lemma:Lower bound of psi}, there is an $\epsilon_2>0$, such that $h_{\lambda_0}>\epsilon_2$ on the set $\Breve{\Omega}_{\lambda_0}\setminus N_{\lambda_0}$. Now by the uniform continuity of $h_{\lambda}$, we may decrease the value of $\delta$ (still make it positive), such that $h_{\lambda}>\frac{\epsilon_2}{2}$ on the set $\Breve{\Omega}_{\lambda}\setminus N_{\lambda}$ for all $\lambda_0\leq \lambda\leq \lambda_0+\delta$.
 \end{proof}
Now we apply Hopf's lemma to prove that $\psi$ and $\Omega_{\Gamma}$ are even symmetric to the $(1,0)$ direction to finish the proof of Proposition \ref{prop:rigidity quadratic}.
 \begin{proof}[Proof of Proposition \ref{prop:rigidity quadratic}]
      Lemma \ref{lemma:interior moving} shows there exists a $\lambda_0>0$ with $h_{\lambda_0}\geq 0$ in $\breve{\Omega}_{\lambda_0}$ and at least one of the internal boundary tangency and boundary tangency happens for 
      $ \breve{\Omega}_{\lambda_0}$.
     In the case where internal tangency occurs at $x^0\in \partial \Omega_{\Gamma}$, by Lemma \ref{claim:coefficient estimate}, $h_{\lambda_0}$ is a non-negative solution of the equation \eqref{problemlambda0}.
      Moreover, it holds that
      \[
      \nabla h_{\lambda_0}(x^0)=\nabla \psi(x^0)-\nabla \psi_{\lambda_0}(x^0)=0.
      \]
      Now we apply Proposition \ref{Hopf Lemma 1} for a sufficiently small ball $B\subset \breve{\Omega}_{\lambda_0}$ that hits $\breve{\Omega}_{\lambda_0}$ internal-tangentially at $x^0$, and conclude $h_{\lambda_0}=0$. Hence $\psi$ and correspondingly $\Omega_{\Gamma}$ is even symmetric in the direction $(1,0)$.
      
     Now we consider the case where the boundary of tangency happens at $x^0\in H_{\lambda_0}\cap \partial \Omega_{\Gamma}$. First, similar to above, one has $\nabla h_{\lambda_0}(x^0)=0$. In addition, $\partial \breve{\Omega}_{\lambda_0}$ forms a right angle at $x^0$, and the normal direction of $\partial \Omega_{\Gamma}$ at $x^0$ is parallel to direction $(0,1)$. In particular, we have $\nabla ^2 h_{\lambda_0}(x^0)=0$. Now analogously as in the case of the internal tangency, by applying Proposition \ref{Hopf Lemma 2} in the appendix to a small half ball contained in $\Breve{\Omega}_{\lambda_0}$ that hits $\Omega_{\Gamma}$ boundary-tangentially, we have $h_{\lambda_0}=0$. Hence $\psi$ and correspondingly $\Omega_{\Gamma}$ is even symmetric in the direction $(1,0)$. 
 \end{proof}

 
\subsection{The proof of Theorem \ref{MainTheorem}}\label{section:finish proof}
{In this section, we finish the proof of Theorem \ref{MainTheorem}.}

{We first introduce the following notation.
	\begin{definition}
		Given a set $U\subset \mathbb{R}^2$ and a function $\psi\in C^2(U)$, if there exists a continuous function $F$ such that $\Delta\psi=F(\psi)$ in $U$, then $\psi$ is said to be realizable by $F$ in $U$, or we abbreviate as $\psi$ is realizable in $U$. 
	\end{definition}
}

{If $F$ is extendable in $\Omega_\Gamma$ as in Lemma \ref{lemma: Puisuex expression quadratic}, since $\Delta\psi -F(\psi)$ is an analytic function in a tubular neighborhood of $\partial \Omega_{\Gamma}$ and $\Delta\psi -F(\psi)=0$ in $\Omega_{\Gamma}$, $\psi$ must be realizable  by $F$ in a domain containing $\Omega_\Gamma$.}
	Furthermore,  we have the following proposition.

{\begin{proposition}\label{Cor: innermost critical loop}
		Assume that $U$ is a simply connected bounded open set, $\psi\in C^\omega(\bar{U})$ is constant on ${\partial U}$. {Assume $\psi$ is realizable by $F$ in $U$} and $F$ is not analytic at $c:=\psi (\partial U)$, then $\psi$ is radially symmetric.
	\end{proposition}
}
\begin{proof}
{Without loss of generality, we assume $c=0$.}
{Since $F$ is not analytic at $0$, {$\partial U$} consists of critical points of $\psi$.
{By Proposition \ref{structure of nodal set}, we can find an innermost loop $\Gamma_{0}$} of critical points of $\psi$ where $F$ loses analyticity (i.e., $F$ is analytic at the values of $\psi$ in every {curve} of critical points contained in the open region enclosed by $\Gamma_0$). {Now $\psi$ must be single-signed in the region $\Omega_{\Gamma_0}$ which is enclosed by $\Gamma_0$. Otherwise, there exist two points $x^0,x^1\in \Omega_{\Gamma_0}$ such that $\psi(x^0)<0$ and $\psi(x^1)>0$. By the continuity of $\psi$, there is a zero point of $\psi$ in every continuous curve that connects $x^0$ and $x^1$ and is contained in $\Omega_{\Gamma_0}$. In particular, $\{\psi =0\}$ has infinitely many points in $\Omega_{\Gamma_0}$ and there exists a curve $\Gamma_1 $ consisting of zeros of $\psi$. Since $F$ is not analytic at $0$, $\Gamma_1$ must be a curve of critical points. This contradicts the definition of $\Gamma_0$.} Analogously as in Section \ref{sec3.2} and Section \ref{sec4.1}, applying Proposition \ref{prop: rigidity in degenate case} or Proposition \ref{prop:rigidity quadratic} depending on the degree of vanishing of $\psi$ in $\Gamma_{0}$,  $\psi$ must be a radially symmetric function.}
\end{proof}

To prove Theorem \ref{MainTheorem}, we need the following important lemma.
{\begin{lemma}\label{lemma:realizable}
		Assume that $U$ is a simply connected bounded open set, $\psi\in C^\omega(\bar{U})$  is constant on $\partial U$, and $\{\psi,\Delta\psi\}=0$ in $U$. If $\psi$ is neither radially symmetric nor realizable in $U$,  then there exists a simply connected domain $\tilde{U} \varsubsetneq U$ such that $\nabla\psi=0$ on $\partial\tilde{U}$ and $\psi$ is not realizable in $\tilde{U}$. 
	\end{lemma}
}
\begin{proof}
{{Suppose that there exists an $x\in U$ such that $\nabla \psi(x)\neq 0$, otherwise, $\psi$ is constant  and thus realizable in $U$.}
		 Then, by Lemma \ref{lemma:function relation bracket case}, $\psi$ is realizable by some $F$ in $B_\epsilon(x)$ for some small $\epsilon.$ Now let $U_{0}$ be the largest connected set containing $B_{\epsilon}(x)$, in which $\psi$ is realizable. By analyticity of $\psi$, $\psi$ is relizable by $F$ in $U_0$. Since $\psi$ is not realizable in $U$, $U_0$ must be a proper subset of $U$.
	}
	
{We claim that $U_0$ cannot be simply connected. Indeed, if $U_0$ were simply connected, then  $\partial U_0$ is a loop of critical points of $\psi$ and that $F$ is not analytic at $c=\psi(\partial{U}_0)$. It follows from Proposition \ref{Cor: innermost critical loop} that $\psi$ must be radially symmetric. This leads to a contradiction.
	}

{{We now take $\tilde{U}$ to be a simply-connected component of the set $\Omega\setminus U_0$. $\tilde{U}$ would be a simply connected domain, and the boundary $\tilde{\Gamma}$ is a piesewise analytic curve. Moreover, assume $\psi$ obtain  value $\tilde c$ at $\Tilde{\Gamma}$, then the $F$ ($F$ is previously defined in the range of $\psi$ in $\overline{U_0}$ and we have $\Delta \psi =F(\psi)$ in $\overline{U_0}$.)  is not analytic at the value $\tilde c$.} Now, suppose $\psi$ were realizable in $\tilde{U}$, say by $G$. We claim that $G$ is not analytic at $\tilde c.$ Otherwise, let $\tilde{G}$ be the analytic extension of $G$. As $\Delta \psi -\tilde{G}(\psi)$ is an analytic function in a tubular neighborhood $\mathcal{N}$ of $\tilde{\Gamma}$ and $\Delta \psi -\tilde{G}(\psi)=0 $ in $\tilde{U}$, we have $\Delta \psi -\tilde{G}(\psi)=0$ in $\mathcal{N}$. This, in particular, implies that $F$ is analytic at $\tilde{c}$, which leads to a contradiction. Since $G$ is not analytic at $\tilde{c}$, we could apply Proposition \ref{Cor: innermost critical loop} to $\tilde{U}$ and conclude that $\psi$ is radially symmetric. Therefore, $\psi$ is not realizable in $\tilde{U}$. Clearly, by definition,  $\tilde{U}\varsubsetneq U$.  This completes the proof of the lemma.
	}
\end{proof}
{Based on Lemma \ref{lemma:realizable}, we give the proof of Theorem \ref{MainTheorem}. 
}
\begin{proof}[Proof of Theorem \ref{MainTheorem}]
{Suppose $\psi$ is neither radially symmetric nor realizable in $\Omega_0:=\Omega$, applying Lemma \ref{lemma:realizable} iteratively, we obtain a nested family of simply connected domains $\{\Omega_{n}\}_{n=1}^\infty$ such that for $n\geq 1$, $\Omega_{n}\subsetneq \Omega_{n-1}$, $\psi$ is not realizable in  $\Omega_{n}$, and $\nabla \psi =0$ on $\partial \Omega_{n}$. This, in particular, implies that there is an infinite family of nested loops of critical points for the analytic function $\psi$. It follows from Proposition \ref{structure of nodal set} that $\psi$ is a constant function. This contradiction implies that $\psi$ must be radially symmetric or realizable in $\Omega$. 
}
\end{proof}
\subsection{Counterexample}\label{sec: counterexample}
It is important to note that the rigidity in Theorem \ref{MainTheorem} is optimal in terms of regularity in bounded domains. First, the results of \cite{enciso2024smooth} show that Theorem \ref{MainTheorem} does not hold for flows that have finite smoothness, even for the weaker assumption of non-locally-radial flows. Secondly, Theorem \ref{MainTheorem} does not hold in the case where the compact domain $\Omega$ is not simply connected, as the following result shows. 
\begin{proposition}\label{main3}
Fix an analytic Jordan curve $\gamma\subset \mathbb{R}^2$. Then there exist a tubular neighborhood $\Omega:=\Omega^{int} \cup \gamma \cup\Omega^{ext}$ of $\gamma$ and a stream function $\psi\in C^\omega(\bar\Omega)$, which is constant on $\partial\Omega$ and solves the following problem. 
\[ \begin{cases}
\Delta \psi=2+\psi^{\frac{5}{2}} \text{ in }\Omega^{int},\\
              \Delta \psi=2-\psi^{\frac{5}{2}} \text{ in } \Omega^{ext},
              \end{cases} \text{ and }\psi=\frac{\partial \psi}{\partial n}=0 \text{ on }\gamma.  
\]
\end{proposition}

Thus, any non-circular $\gamma$ gives rise to a steady Euler flow $\psi$,  that is analytic in a neighborhood of $\gamma$, and  that is neither radial nor a solution of a semilinear elliptic equation of the form \eqref{SemilinearElliptic}. 

{In order to prove Proposition~\ref{main3}, we construct an analytic function $\psi$ in a tubular neighborhood $\overline{\Omega^{int} \cup \Omega^{ext}}$ of $\gamma$ satisfying 
\[
\Delta \psi = 2 + \psi^{\frac{5}{2}} \quad \text{in } \Omega^{int}.
\]
The construction of such an analytic function is inspired by the classical proof of the Cauchy--Kovalevskaya Theorem, where we apply the so-called majorization method to a carefully chosen majorizing system, so that the existence of the analytic solution is established.}

{
Once such an analytic function $\psi$ is obtained, we can verify that $\{\psi,\Delta \psi\}=0$ in $\Omega^{ext}$. Analogous to the proof of Lemma \ref{lemma:Puiseux expression near 1}, we can then prove
\[
\Delta \psi = 2 - \psi^{\frac{5}{2}} \quad \text{in } \Omega^{ext}.
\]  A detailed proof can be found in Sections~2.4.3--2.4.5 of the thesis of the second-named author \cite{yupeithesis}.
}

\section{Conclusion and Future Directions}

We gather here a few interesting directions and problems that can be considered in view of Theorem \ref{MainTheorem} and its proof. We have already mentioned that replacing the analyticity assumption with some other less stringent assumption would be desirable. We mention some other interesting problems below. 

\subsection{Local and Global Questions}
Given Theorem \ref{MainTheorem} and the extremely large literature on semilinear elliptic equations, an important problem is to study the local and global properties of the set of Euler steady states. As mentioned in the introduction, the invertibility properties of the linear Schr\"odinger operator \[\mathcal{L}=\Delta-F'(\psi_*)\] play an important role in the local structure of steady states (see \cite{choffrut2012local,constantin2021flexibility,coti2023stationary,danielski2024complex, lin2011inviscid}). One consequence of Theorem \ref{MainTheorem} and the proof is that we can always define $\mathcal{L}$ as a (unbounded) self-adjoint operator on $L^2$. This gives hope for a general local understanding of the manifold of steady states and then possibly a global one.

Most of the previous results on the local structure of steady states have assumed the invertibility of $\mathcal{L}$. It is an interesting problem to study the general case where $\mathcal{L}$ has a kernel. From the classical Fredholm theory, we know that the kernel is finite-dimensional. Once $\mathcal{L}$ has a kernel, the local structure of steady states may be quite complicated \cite{coti2023stationary}.  It is an interesting problem to study this in general, and this is related to the (im)possibility of the existence of isolated states \cite{drivas2023singularity}. If we had a good understanding of the local structure of steady states in the general case, we could hope to patch together this information to get more global information about the structure of the set of analytic Euler steady states. Some of these issues could be simpler to study in generic domains \cite{uhlenbeck1976generic}. 
 In general, it would be interesting to have results related to the connectedness, path-connectedness, and smoothness of the manifold of steady states (see \cite{danielski2024complex,danielski2024analytic}).
 
\subsection{Rigidity Questions on Other Domains}
 A first obvious question is whether similar results can be established in different domains. We have already shown that it does not hold in some annular domains, although it is not clear whether a modification holds on $\mathbb{T}\times [0,1],$ for example. An interesting problem is to determine whether a version of Theorem \ref{MainTheorem} can hold on $\mathbb{T}^2.$ In that case, the question would be whether all analytic steady states are necessarily either a shear flow or a solution to a semilinear elliptic equation. 

\subsection*{Acknowledgements}

T. M. E. acknowledges partial funding from the NSF DMS-2043024, the Alfred P. Sloan foundation, and the Simons Foundation. He also acknowledges helpful conversations with T. Drivas and A. Kiselev. Y. H. acknowledges partial funding from NSF DMS-2043024. A. R. S. is supported by the ERC/UKRI advanced grant SWAT. 
The research of  C. X. is partially supported by  NSFC grant 12250710674, and Program of Shanghai Academic Research Leader 22XD1421400.
 The authors thank the referees for many helpful comments that significantly improved the presentation of this paper.
\appendix

\section{Hopf Lemmas for Elliptic Equations with Singular Coefficients}
In this section, we list some known facts about the maximum principle for solutions to the elliptic equation $-\Delta h+c(x)h=0$ in a Lipschitz domain $\Omega$. Here $c(x)$ could be unbounded but satisfies the following uniform upper bound $
   c^{-}(x)\leq \frac{M}{dist(x,\partial\Omega)}$.
The results we present here are crucial in our proof of Proposition \ref{prop:rigidity quadratic}.
\begin{proposition}[The Hopf Lemma in small balls, see Proposition 4.3 in \cite{ruiz2023symmetry}]\label{Hopf Lemma 1}
    We consider a non-negative smooth solution $h$ to the following elliptic equation in the ball $B_{r}(0,r):$ \begin{equation}
       -\Delta h+ c(x)h =0.
   \end{equation} Now, if we further assume $h(0)=0$, there exists a $r_0>0$ depending on $M$, such that if $r<r_0(M)$, we have either $h=0$ or $\frac{\partial h}{\partial x_2}(0)>0$. 
\end{proposition}

\begin{proposition}[The Hopf Lemma in small half balls, see Proposition 4.4 in \cite{ruiz2023symmetry}]\label{Hopf Lemma 2}
  Let $B_{r}^{+}:=B_{r}(0,-r)\cap \{x_1\geq 0\}$, we consider a non-negative smooth solution $h$ to the following elliptic equation in $B_{r}^{+}:$ \begin{equation}
       -\Delta h+ c(x)h=0 \text{ with $c^{-}(x)\leq \frac{M}{dist(x,0)}$.}
   \end{equation} Now, if we further assume $\nabla h(0)=h(0)=0$, there exists a $r_0>0$ depending on $M$, such that if $r<r_0(M)$, we have either $h=0$ or $\frac{\partial^2 h}{\partial \eta  ^2}(0)>0$, for any direction $\eta=(\eta_1,\eta_2)$ satisfying $\eta_1>0$, $\eta_2<0$ in $B_{r}^{+}$. 
\end{proposition} 
\section{On the Rigidity of Steady States with Single-Signed Vorticity}\label{application appendix}
We now proceed to make precise statements in Remark \ref{convex mirror}. 
\begin{proposition}\label{convex rigidity}
    Let $\Omega$ be a simply connected, bounded, and convex domain, and $\psi$ be the stream function of a steady Euler flow with non-positive vorticity in $\Omega$. If we further assume that $\Omega$ is even symmetric in the direction $e:=(1,0)$, and $\psi$ is analytic, then $\psi$ is even symmetric in the direction $e$.  
\end{proposition}
The full proof of this theorem follows \textit{inverbatim} from the proof of Theorem \ref{MainTheorem}. We explain the key ideas here.  In the case where $\psi$ is nonradial, by Theorem \ref{MainTheorem}, there exists an $F$ such that $\Delta \psi=F(\psi)$ in $\Omega$. Without loss of generality, we assume $\psi =0$ in $\partial \Omega$.  As $\Delta \psi$ is non-positive, unless $\psi=0$, we have $\psi>0$ in $\Omega$ and  $F$ is analytic near $0$. Without loss of generality, $max_{\Omega} \psi=1$, then $F$ is analytic in $(0,1)$ and admits a Puiseux expansion near $1$ as Lemma \ref{lemma:Puiseux expression near 1}. Analogously to the proof of Lemma \ref{lemma:interior moving}, we can show $\psi$ is even symmetric in direction $e$.
\begin{remark}
An interesting question is to extend the study in Proposition \ref{convex rigidity} when the domain is non-convex. More precisely, in bounded simply connected domains with mirror symmetry, determine whether analytic steady flows with single-signed vorticity should enjoy mirror symmetry.  
\end{remark}

\renewcommand{\baselinestretch}{1.25}

\bibliographystyle{plain}
\bibliography{biblio}
\end{document}